\documentclass{svjour3}                     
\smartqed  
\usepackage{setspace}
\usepackage{algorithm}
\usepackage{algorithmic}

\usepackage{amssymb}
\usepackage{booktabs}
\usepackage{graphicx,fancyhdr}
\usepackage{graphics,color}
\usepackage{subfigure}
\usepackage{booktabs}
\usepackage{graphicx}
\usepackage{amsmath}
\usepackage{graphicx}
\usepackage{amsmath}
\usepackage{epstopdf} 
\usepackage{bbding}
\usepackage{mathptmx}
\usepackage{algorithm}
\usepackage{algorithmic}
\usepackage{color}
\numberwithin{equation}{section}
\numberwithin{theorem}{section}
\numberwithin{lemma}{section}
\numberwithin{remark}{section}

\usepackage{mathptmx}
\begin{document}

\title{Uniform convergence of  V-cycle multigrid algorithms for two-dimensional fractional Feynman-Kac equation}



\author{Minghua Chen         \and
         Weihua Deng          \and
 Stefano Serra-Capizzano
}

\institute{
 M. Chen   (\Envelope)   \and W. Deng        \\
School of Mathematics and Statistics, Gansu Key Laboratory of Applied Mathematics and Complex Systems, Lanzhou University, Lanzhou 730000, P.R. China\\
email:chenmh@lzu.edu.cn;   dengwh@lzu.edu.cn\\
S. Serra-Capizzano \\
Department of Science and High Technology, University of Insubria, Via Valleggio 11, 22100 Como, Italy $\&$ Department of Information Technology, Division of Scientific Computing, Uppsala University - ITC, Lägerhyddsv. 2, hus 2,  P.O. Box 337, SE-751 05, Uppsala, Sweden\\
email: stefano.serrac@uninsubria.it, stefano.serra@it.uu.se }
\date{Received: date / Accepted: date}

\maketitle

\begin{abstract}
When solving large linear systems stemming from the approximation of elliptic partial differential equations (PDEs), it is known that the V-cycle multigrid method (MGM)  can significantly lower the computational cost.
Many convergence estimates already exist for   the V-cycle MGM: for example,
using the regularity or approximation assumptions  of the elliptic PDEs, the results are obtained  in
[Bank \&  Douglas, SIAM J. Numer. Anal. \textbf{22}, 617-633  (1985); Bramble  \&   Pasciak, Math. Comp. \textbf{49},  311-329  (1987)]; in the case of  multilevel matrix algebras (like circulant, tau, Hartely)
[Aric\`{o}, Donatelli  \&   Serra-Capizzano, SIAM J. Matrix Anal. Appl.  \textbf{26}, 186-214  (2004); Aric\`{o} \&  Donatelli, Numer. Math. \textbf{105}, 511-547  (2007)], special prolongation operators are provided and the related convergence results are rigorously developed, using a functional approach.
In this paper we  derive  new uniform convergence estimates for the V-cycle MGM applied to symmetric positive definite Toeplitz block tridiagonal matrices, by also discussing few connections with previous results.
More concretely, the contributions of this paper are as follows: (1) It tackles the  Toeplitz systems directly for the elliptic PDEs.
(2)  Simple (traditional) restriction operator and prolongation operator are employed in order to handle general Toeplitz systems at each level of the recursion.
Such a technique is then applied to systems of algebraic equations generated by the difference scheme of the two-dimensional fractional Feynman-Kac equation, which describes the joint probability density function of non-Brownian  motion. In particular,  we consider the two coarsening strategies, i.e.,  doubling the mesh size (geometric MGM) and  Galerkin approach  (algebraic MGM),  which lead to the distinct  coarsening stiffness matrices in the general case: however, several numerical experiments show that the two algorithms produce almost the same error behaviour.
\end{abstract}


\section{Introduction}
When considering iterative solvers for large linear systems stemming from the approximation of partial differential equations (PDEs),  multigrid methods (MGM) (such as backslash cycle, V-cycle and W-cycle) have often been shown to provide algorithms with optimal order of complexity \cite {Bramble:87,Hackbusch:85}.
Using  the regularity or approximation assumptions of the elliptic PDEs, the complete proof on the uniform convergence of the MGM for second order elliptic equation has been discussed in  \cite{Bank:85,Bramble:87} and several outstanding works have been derived in this direction, e.g., \cite{Bramble:91,Brenner:08,Xu:02}.
On the hand, concerning linear systems with coefficient matrix belonging to multilevel matrix algebras (like circulant, tau, Hartely), the proof of convergence  of  the  two-grid methods are given in  \cite{Arico:04,Arico:07,Serra-Capizzano:02} and the level independence is discussed in \cite{Serra-Capizzano:02},  for  special prolongation operators  \cite{Fiorentino:91,Fiorentino:96} associated to the symbol of the coefficient matrices; moreover, the uniform convergence  of the  V-cycle MGM is further derived in \cite{Arico:04} and extended in \cite{Arico:07,Bolten:15} for the elliptic Toeplitz and PDEs matrices.
In recent years, the multigrid methods have also been applied to solve the fractional differential equations (FDEs) \cite{Chen:14,ChenDeng:15,Pang:12}; for  time-dependent FDEs \cite{Chen:14,Pang:12}, the two-grid method is used and the convergence analysis is performed by following the ideas in \cite{Chan:98,Fiorentino:96}, in which different prolongation operators are required at each recursion level, when dealing with general Toeplitz systems.
In this paper, we use the simple (traditional)  restriction operator and prolongation operator to handle general Toeplitz systems directly for the elliptic PDEs.
Then we  derive  new   uniform convergence  estimates regarding the V-cycle MGM for symmetric positive definite Toeplitz block tridiagonal matrices, which can be  applied to   the fractional Feynman-Kac (FFK) equation \cite{Carmi:10,Turgeman:09}.
Regarding numerical experiments, we consider two coarsening strategies for MGM. The first is based on simple coarsening strategy, i.e., doubling the mesh size ($h\rightarrow 2h$) in each spatial direction, leading to the so called geometric MGM: in this case
the coarse stiffness matrix is the natural analog of  the finest grid coefficient matrix. The second strategy is based on the  Galerkin approach and is refereed to as algebraic MGM \cite{Bolten:15,Trottenberg:01}. From the basic theoretical point of view, the major advantage of Galerkin approach is that it satisfies the variational principle; however, from the practical point of view, we find that they  almost lead to the same numerical results.

 After obtaining the uniform convergence  for the V-cycle MGM, we apply it to the  difference scheme for the backward fractional Feynman-Kac equation \cite{Carmi:10}, which describes the distribution of the functional of the trajectories of non-Brownian motion,
 defined by $U\rightarrow A(U)=\int_0^t U[{\bf x}(\tau)]d\tau$. There are many special or interesting choices for $U({\bf x})$, e.g.,
taking $U({\bf x})=1$ in a given domain and zero otherwise, this functional can be used in kinetic studies of chemical reactions that take place exclusively in the domain \cite{Bar-Haim:98,Carmi:10}. For inhomogeneous disorder dispersive systems, the motion of the particles is non-Brownian, and $U({\bf x})$ is taken as ${\bf x}$ or ${\bf x}^2$ \cite{Carmi:10}.
The multi-dimensional backward fractional Feynman-Kac equation is given as \cite{Carmi:10,Turgeman:09}
\begin{equation}\label{1.1}
\frac{\partial}{\partial t}G({\bf x},\rho,t)=\kappa_\alpha \, {^s\!}D_t^{1-\alpha}\Delta G({\bf x},\rho,t)-\rho U({\bf x})G({\bf x},\rho,t)
\quad \forall {\bf x} \in  \mathbb{{R}}^n,
\end{equation}
where  $G({\bf x},\rho ,t)=\int_0^\infty G({\bf x},A,t)e^{-\rho A}dA$, $Re (\rho)>0$, $U({\bf x})>0$, the diffusion coefficient $\kappa_\alpha $ is a positive constant and $\alpha \in (0,1)$, and the Riemann-Liouville fractional substantial derivative is defined by \cite{Chen:13}
 \begin{equation*}
{^s\!}D_t^\alpha G({\bf x},\rho,t)={^s\!}D_t^m[{^s\!}I_t^{m-\alpha} G({\bf x},\rho,t)],
\end{equation*}
with the fractional substantial integral ${^s\!}I_t^\beta $ ($\beta>0$) expressed as
\begin{equation*}
{^s\!}I_t^\beta G({\bf x},\rho,t)=\frac{1}{\Gamma(\beta)}\int_{0}^t{\left(t-\tau\right)^{\beta-1}}e^{-\rho U({\bf x})(t-\tau)}{G({\bf x},\rho,\tau)}d\tau, ~~~~t>0.
\end{equation*}
Similarly, we can  define the Caputo fractional substantial derivative of order $\alpha$ as
\begin{equation*}
 {_c^sD}_t^\alpha G({\bf x},\rho,t)={^s\!}I_t^{m-\alpha}[{^s\!}D_t^m G({\bf x},\rho,t)].
\end{equation*}
Then (\ref{1.1}) can be rewritten in the form \cite{Deng:14}
\begin{equation}\label{1.2}
\begin{split}
{^s_c}{D}_t^\alpha G({\bf x},\rho,t)
&={^s\!}D_t^\alpha [G({\bf x},\rho,t)-e^{-\rho t}G({\bf x},\rho,0)]\\
&={^s\!}D_t^\alpha G({\bf x},\rho,t)-\frac{t^{-\alpha}e^{-\rho  t}}{\Gamma(1-\alpha)}G({\bf x},\rho,0)
  = \kappa_\alpha \Delta G({\bf x},\rho,t).
\end{split}
\end{equation}

The  outline of the paper is as follows. In  the next section,  we derive the convergence estimates of the V-cycle MGM
for the symmetric positive definite Toeplitz tridiagonal matrix. For symmetric positive definite Toeplitz block tridiagonal matrix,  the convergence estimates of the V-cycle MGM are given in Section 3.
In Section 4, we present the compact difference scheme for (\ref{1.2}) in 1D, and the centered difference scheme for
(\ref{1.2}) in 2D.  Then in Section 5, we use the presented  V-cycle MGM framework for the efficient computational solution of the resulting algebraic systems of linear equations. Results of numerical experiments are reported and discussed in Section 6, in order to show the effectiveness of the presented schemes. Finally, we conclude the paper with some remarks.

\section{Uniform convergence  of V-Cycle MGM  for 1D}
Let us first consider the simple algebraic system (1D)
\begin{equation}\label{2.1}
  A_h\nu^h=f_h,
\end{equation}
where
$$A_h={\rm tridiag}(a_1, a_0, a_1)~~ {\rm with}~~a_0\geq 2|a_1|~~ {\rm and}~~a_0>0.$$

Let $\Omega \in (0,b)$ and  the mesh points $x_i=ih$, $h=b/(M+1)$.
To describle  the MGM, we need to define the following  multiple level of grids
\begin{equation}\label{2.2}
  \mathcal{B}_k=\Big\{x_i^k=\frac{i}{2^k}b,\, i=1: M_k \Big\} ~~{\rm with}~~M_k=2^k-1,\, k=1 : K,
\end{equation}
where $\mathcal{B}_K=\mathcal{B}_h$ is the finest mesh and $M=2^K-1$.
We adopt the notation that $\mathcal{B}_k$ represents not only the grid with grid spacing $h_k=2^{(K-k)}h$, but also the space of vectors defined on that grid.
For the one dimensional  system, the  restriction operator $I_k^{k-1}$ and prolongation operator $I_{k-1}^k$ are, respectively, defined by \cite[p.\,438-454]{Saad:03}
 \begin{equation}\label{2.3}
\begin{split}
\nu^{k-1}=I_k^{k-1}\nu^k ~~{\rm with}~~ \nu_i^{k-1}=\frac{1}{4}\left(\nu_{2i-1}^{k}+2\nu_{2i}^{k}+\nu_{2i+1}^{k}\right),
~~~i=1:M_{k-1},
\end{split}
\end{equation}
and
 \begin{equation}\label{2.4}
\begin{split}
\nu^{k}=I_{k-1}^{k}\nu^{k-1}~~{\rm with}~~ I_{k-1}^{k} =2\left(I_k^{k-1}\right)^T,
\end{split}
\end{equation}
where
\begin{equation}\label{2.5}
I_{k-1}^k=\frac{1}{2}\left [ \begin{matrix}
                      1   &     &          & \\
                      2   &     &          &  \\
                      1   &  1  &          &   \\
                          &  2  &          &    \\
                          &  1  &  \ddots  &     \\
                          &     &  \ddots  &      \\
                          &     &  \ddots  & 1     \\
                          &     &          & 2      \\
                          &     &          & 1       \\
 \end{matrix}
 \right ]_{M_{k}\times M_{k-1}}.
\end{equation}
The coarse problem is typically defined by the Galerkin approach
\begin{equation}\label{2.6}
  A_{k-1}=I_k^{k-1}A_kI_{k-1}^{k},
\end{equation}
and the intermediate $(k,k-1)$ coarse grid correction operator is
\begin{equation}\label{2.7}
  T^k=I_k- I_{k-1}^{k}A_{k-1}^{-1}I_k^{k-1}A_k=I_k- I_{k-1}^{k}P_{k-1}
\end{equation}
with
\begin{equation*}
  P_{k-1}=A_{k-1}^{-1}I_k^{k-1}A_k.
\end{equation*}
Let $K_k$ be the iteration matrix of the smoothing operator. In this work, we
 take $K_k$ to be  the weighted (damped) Jacobi iteration matrix
\begin{equation}\label{2.8}
  K_k=I-S_kA_k,~~{\rm where}~~S_{k}:=S_{k,\omega}=\omega D_k^{-1}
\end{equation}
with a weighting  factor $\omega \in (0,1/2]$, and $D_k$ is the diagonal of $A_k$.

A multigrid process can be regarded as defining a sequence of operators $B_k:\mathcal{B}_k\mapsto \mathcal{B}_k$
which is an approximate inverse of $A_k$ in the sense that $||I-B_kA_k||$ is bounded away from one.
The V-cycle multigrid algorithm \cite{Bramble:87} is provided in Algorithm \ref{MGM}.
\begin{algorithm*}
\caption{ V-cycle Multigrid Algorithm: Define $B_1=A_1^{-1}$. Assume that $B_{k-1}:\mathcal{B}_{k-1}\mapsto \mathcal{B}_{k-1}$ is defined.
We shall now define $B_k:\mathcal{B}_{k}\mapsto \mathcal{B}_{k}$ as an approximate iterative solver for the equation  $A_k\nu^k=f_k$.}
\label{MGM}
\begin{algorithmic}[1]
\STATE Pre-smooth: Let $S_{k,\omega}$ be defined by (\ref{2.8}), $\nu^k_0=0$, $l=1: m_1$, and
  $$\nu^k_l=\nu^k_{l-1}+S_{k,\omega_{pre}}(f_k-A_k\nu^k_{l-1}).$$
\STATE Coarse grid correction: Denote $e^{k-1} \in \mathcal{B}_{k-1}$ as the approximate solution of the residual equation $A_{k-1}e=I_k^{k-1}(f_k-A_k\nu^k_{m_1})$
with the iterator $B_{k-1}$:
$$e^{k-1}=B_{k-1}I_k^{k-1}(f_k-A_k\nu^k_{m_1}).$$
\STATE Post-smooth:~~$\nu^k_{m_1+1}=\nu^k_{m_1}+I_{k-1}^{k}e^{k-1}$, $l=m_1+2: m_1+ m_2$, and
$$\nu^k_l=\nu^k_{l-1}+S_{k,\omega_{post}}(f_k-A_k\nu^k_{l-1}).$$
\STATE Define $B_kf_k=\nu^k_{m_1+m_2}$.
\end{algorithmic}
\end{algorithm*}

Since the matrix $A:=A_h$ is symmetric positive definite,  we can define the following inner products \cite[p. 78]{Ruge:87}
\begin{equation}\label{2.9}
  ( u,v  )_D=(Du,v), \quad (u,v)_{A}=(Au,v), \quad (u,v)_{AD^{-1}A}=(Au,Av)_{D^{-1}},
\end{equation}
where $(\cdot,\cdot)$ is the usual Euclidean inner product.
Here the finest grid operator is $A_h$ or $A_K$  with the finest grid size $h$;
and the coarse grid operators   $A_{k-1}=I_k^{k-1}A_kI_{k-1}^{k}$ are defined by the Galerkin approach (\ref{2.6})   with the grid sizes $\{2^{K-k}h\}_{k=1}^{K-1}$.

\subsection{Improved framework for the MGM}
Based on the framework of \cite{Bramble:87,Xu:02},  we now present the  estimates on the convergence rate of the MGM, namely,
\begin{equation*}
||I-B_kA_k||_{A_k}<1,
\end{equation*}
where $I$ is identity matrix and $A_k$, $B_k$ are given in   Algorithm  \ref{MGM}.

Assume that the following two assumptions are satisfied, i.e.,
\begin{equation}\label{2.10}
\frac{\omega}{\lambda_{\max}(A_k) }(\nu^k,\nu^k) \leq (S_k\nu^k,\nu^k)\leq (A_k^{-1}\nu^k,\nu^k)~~~\forall \nu^k\in \mathcal{B}_k,
\end{equation}
and
\begin{equation}\label{2.11}
||T^k\nu^k||^2_{A_k}\leq m_0 ||A_k\nu^k||_{D_k^{-1}}^2~~\forall \nu^k\in \mathcal{B}_k,
\end{equation}
where  $\omega$ is defined by (\ref{2.8}).
For the complete proof on the uniform convergence of  the MGM,  there exists the following lemma.
\begin{lemma} [\cite {Bramble:87,Xu:02}]\label{lemma2.1}
If  $A_k$ satisfies (\ref{2.10}) and (\ref{2.11}),  then
$$||I-B_kA_k||_{A_k} \leq \frac{m_0}{2l\omega+m_0}<1~~{\rm with}~~~1\leq k\leq K,$$
where  the operator $B_k$ is defined by the  V-cycle method in   Algorithm  \ref{MGM}  and  $l$ is the number of smoothing steps.
\end{lemma}

It is well known that the framework of the convergence analysis of the MGM \cite{Bramble:87,Xu:02}
 is  based on the verification (\ref{2.10}) and (\ref{2.11}).
However, it is not at all easy to prove the assumption (\ref{2.11}) in general, since it needs to  solve $A_{k-1}^{-1}$ in (\ref{2.7}).
Here, we replace  the condition (\ref{2.11}) by the following Lemma, which simplifies the theoretical investigations substantially.

\begin{lemma}\label{lemma2.2}
Let $A_k$ be a symmetric positive definite matrix and
\begin{equation}\label{2.12}
   \min_{\nu^{k-1} \in \mathcal{B}_{k-1} }||\nu^k-I_{k-1}^k\nu^{k-1}||_{A_k}^2\leq m_0 ||A_k\nu^k||_{D_k^{-1}}^2 \quad  \forall \nu^k \in \mathcal{B}_k
\end{equation}
with  $m_0>0$ independent of $\nu^k$. Then
\begin{equation*}
||T^k\nu^k||^2_{A_k}\leq m_0 ||A_k\nu^k||_{D_k^{-1}}^2~~\forall \nu^k\in \mathcal{B}_k.
\end{equation*}
\end{lemma}
\begin{proof}
From (\ref{2.12}) and the variational principle for coarse grid operator $T^k$ (see the corollary of   \cite[p.\,431]{Trottenberg:01}), we obtain
\begin{equation*}
  ||T^k\nu^k||^2_{A_k}=  \min_{\nu^{k-1} \in \mathcal{B}_{k-1} }||\nu^k-I_{k-1}^k\nu^{k-1}||_{A_k}^2\leq m_0 ||A_k\nu^k||_{D_k^{-1}}^2.
\end{equation*}
The proof is completed.
\end{proof}

Using Lemmas \ref{lemma2.1} and \ref{lemma2.2}, we have
\begin{theorem}\label{theorem2.3}
If  $A_k$ satisfies (\ref{2.10}) and (\ref{2.12}),  then
$$||I-B_kA_k||_{A_k} \leq \frac{m_0}{2l\omega+m_0}<1~~{\rm with}~~~1\leq k\leq K,$$
where  the operator $B_k$ is defined by the  V-cycle method in   Algorithm  \ref{MGM}  and  $l$ is the number of smoothing steps.
\end{theorem}
\subsection{Convergence estimates  of MGM for  1D}
We now give a complete proof on the uniform convergence of the MGM for the algebraic system (\ref{2.1}),
i.e., we need to examine  the two assumptions (\ref{2.10}) and (\ref{2.12}).

\begin{lemma}\label{lemma2.4}
Let $A^{(1)}=\{a_{i,j}^{(1)}\}_{i,j=1}^{\infty}$ with $a_{i,j}^{(1)}=a_{|i-j|}^{(1)}$ be a  symmetric  Toeplitz  matrix
and  $A^{(k)}=L_h^{H}A^{(k-1)}L_{H}^{h}$ with $L_h^{H}=4I_k^{k-1}$ and $L_H^{h}=(L_h^{H})^T$.
Then $A^{(k)}$ can be computed by
\begin{equation}\label{2.13}
\begin{split}
a_0^{(k)}
=&(4C_k+2^{k-1})a_0^{(1)}+\sum_{m=1}^{2\cdot2^{k-1}-1}{_0}C_m^ka_m^{(1)};\\
a_1^{(k)}
=&C_ka_0^{(1)}+\sum_{m=1}^{3\cdot2^{k-1}-1}{_1}C_m^k a_m^{(1)};\\
a_j^{(k)}
=&\sum_{m=(j-2)2^{k-1}}^{(j+2)2^{k-1}-1} {_j}C_m^k a_m^{(1)} \quad \forall j\geq 2 \quad \forall k\geq 2
    \end{split}
\end{equation}
with $C_k=2^{k-2}\cdot\frac{2^{2k-2}-1}{3}$. And
\begin{equation*}
{_0}C_m^k=\left\{ \begin{split}
&8C_k-(m^2-1)(2^k-m) ~~\quad~{\rm for}~~m=1: 2^{k-1};\\
&\frac{1}{3}(2^{k}-m-1)(2^{k}-m)(2^{k}-m+1)
~~{\rm for}~~m=2^{k-1}:2\cdot2^{k-1}-1;
 \end{split}
 \right.
\end{equation*}
${_1}C_m^k=$
\begin{equation*}
\left\{ \begin{split}
&2C_k+m^2\cdot2^{k-1}-\frac{2}{3}(m-1)m(m+1) ~~\quad~{\rm for}~~m=1: 2^{k-1};\\
&2C_k+(2^k-m)^2\cdot2^{k-1}-\frac{2}{3}(2^k-m-1)(2^k-m)(2^k-m+1)\\
& -\frac{1}{6}(m-2^{k-1}-1)(m-2^{k-1})(m-2^{k-1}+1)
~~\quad~{\rm for}~~~~m=2^{k-1}: 2\cdot2^{k-1};\\
&\frac{1}{6}(3\cdot2^{k-1}\!-\!m\!-\!1)(3\cdot2^{k-1}\!-\!m)(3\cdot2^{k-1}\!-\!m+1)
~~{\rm for}~~m=2\cdot2^{k-1}:3\cdot2^{k-1}-1;
 \end{split}
 \right.
\end{equation*}
and for $j\geq 2$,
\begin{equation*}
{_j}C_m^k=
\left\{ \begin{split}
&\varphi_1 ~~\quad~{\rm for}~~m=(j-2)2^{k-1}:(j-1)2^{k-1};\\
&\varphi_2
~~\quad~{\rm for}~~ m=(j-1)2^{k-1}: j2^{k-1};\\
&\varphi_3
~~\quad~{\rm for}~~ m=j2^{k-1}: (j+1)2^{k-1};
\\
&\varphi_4
~~\quad~{\rm for}~~ m=(j+1)2^{k-1}:(j+2)2^{k-1}-1,
 \end{split}
 \right.
\end{equation*}
where
\begin{equation*}
\begin{split}
\varphi_1=\frac{1}{6}(m-(j-2)2^{k-1}-1)(m-(j-2)2^{k-1})(m-(j-2)2^{k-1}+1);
 \end{split}
\end{equation*}
\begin{equation*}
\begin{split}
\varphi_2=&2C_k+(m-(j-1)2^{k-1})^2\cdot2^{k-1}\\
& -\frac{1}{6}(j2^{k-1}-m-1)(j2^{k-1}-m)(j2^{k-1}-m+1)\\
&-\frac{2}{3}(m-(j-1)2^{k-1}-1)(m-(j-1)2^{k-1})(m-(j-1)2^{k-1}+1);\\
 \end{split}
\end{equation*}
\begin{equation*}
\begin{split}
\varphi_3=&2C_k+((j+1)2^{k-1}-m)^2\cdot2^{k-1} \\
&-\frac{1}{6}(m-j2^{k-1}-1)(m-j2^{k-1})(m-j2^{k-1}+1)\\
&-\frac{2}{3}((j+1)2^{k-1}-m-1)((j+1)2^{k-1}-m)((j+1)2^{k-1}-m+1);\\
 \end{split}
\end{equation*}
\begin{equation*}
\begin{split}
\varphi_4=&\frac{1}{6}((j+2)2^{k-1}-m-1)((j+2)2^{k-1}-m)((j+2)2^{k-1}-m+1).
 \end{split}
\end{equation*}
\end{lemma}
\begin{proof}
See the Appendix.
\end{proof}

\begin{corollary}\label{corollary2.5}
Let $A^{(k)}=I_k^{k-1}A^{(k-1)}I_{k-1}^{k}$ with $A^{(1)}={\rm tridiag}(a_1, a_0, a_1)$. Then
$$A^{(k)}={\rm tridiag}(a_1^{(k)}, a_0^{(k)}, a_1^{(k)}),$$
where
$$a_0^{(k)}=\frac{1}{8^{k-1}}\left[\left(4C_k+2^{k-1}\right)a_0+8C_k a_1 \right],$$
and
$$a_1^{(k)}=\frac{1}{8^{k-1}}\left[C_ka_0+\left(2C_k+2^{k-1}\right) a_1 \right].$$
\end{corollary}
\begin{proof}
From Lemma \ref{lemma2.4}, the desired results can be obtained.
\end{proof}
\begin{lemma}\label{lemma2.6}
Let $A^{(1)}:=A_h$ be defined by
(\ref{2.1}) and $A^{(k)}=I_k^{k-1}A^{(k-1)}I_{k-1}^{k}$. Then
\begin{equation*}
\frac{\omega}{\lambda_{\max}(A_k) }(\nu^k,\nu^k) \leq (S_k\nu^k,\nu^k)\leq (A_k^{-1}\nu^k,\nu^k)
\quad \forall \nu^k\in \mathcal{B}_k,
\end{equation*}
where $A_k=A^{(K-k+1)}$, $S_k=\omega D_k^{-1}$, $\omega \in (0,1/2]$ and $D_k$ is the diagonal of $A_k$.
\end{lemma}
\begin{proof}
According to  Corollary \ref{corollary2.5}, we have
\begin{equation}\label{2.14}
\begin{split}
  A^{(k)}
&=\mu_1\cdot{\rm tridiag}(-1, 2, -1)+\mu_2\cdot{\rm tridiag}(1, 2, 1)=: A_1^{(k)}+ A_2^{(k)}
\end{split}
\end{equation}
with
$$\mu_1=\frac{2C_k\left(a_0+2a_1\right)+2^{k-1}\left(a_0-2a_1\right)}{4\cdot 8^{k-1}}>0,$$
and
$$\mu_2=\frac{\left(6C_k+2^{k-1}\right)\left(a_0+2a_1\right)}{4\cdot 8^{k-1}}\geq 0.$$
Taking  $A^{(k)}=\{a_{i,j}^{(k)}\}_{i,j=1}^{\infty}$, $a_{i,j}^{(k)}=a_{|i-j|}^{(k)} \,~~~\forall k \geq 1$
and using (\ref{2.14}), we obtain
$$ r_i^{(k)}:= \sum\limits_{j\neq i} |a_{i,j}^{(k)}| <  a_{i,i}^{(k)}.$$
From the Gerschgorin circle theorem \cite[p.\,388]{Horn:13}, the eigenvalues of  $A^{(k)}$ are in the disks centered at $a_{i,i}^{(k)}$ with radius
$r_i^{(k)}$, i.e.,  the eigenvalues $\lambda$ of the matrix  $A^{(k)}$ satisfy
$$  |\lambda -a_{i,i}^{(k)} | \leq r_i^{(k)}, $$
which yields
$\lambda_{\max}(A^{(k)}) \leq a_{i,i}^{(k)}+r_i^{(k)}< 2a_{i,i}^{(k)}=2a_{1,1}^{(k)}. $

On the other hand, using the  Rayleigh theorem  \cite[p.\,235]{Horn:13}, i.e.,
$$\lambda_{\max}(A^{(k)})=\max_{x\neq 0}\frac{x^TA^{(k)}x}{x^Tx}\quad\forall x \in  \mathbb{{R}}^n,$$
if we take $x=[1,0,\ldots,0]^T$, it means that
$$\lambda_{\max}(A^{(k)})\geq \frac{x^TA^{(k)}x}{x^Tx}=a_{1,1}^{(k)}.$$
Hence, we obtain
$$\lambda_{\max}\left(\left(D^{(k)}\right)^{-1}A^{(k)}\right)=\frac{\lambda_{\max}(A^{(k)})}{a_{1,1}^{(k)}} \in [1,2),$$
where  $D^{(k)}$ is the diagonal of $A^{(k)}$. It yields
$$1\leq \lambda_{\max}(D_k^{-1}A_k)<2 \quad \forall \nu^k\in \mathcal{B}_k.$$
The proof is completed.
\end{proof}

\begin{lemma}\label{lemma2.7}
Let $L_a={\rm tridiag}(b, a, b)$ and $L_c={\rm tridiag}(d, c, d)$ be symmetric positive definite.
Then  $L_aL_c$ is symmetric positive definite.
\end{lemma}
\begin{proof}
Since   $L_aL_c$ is a symmetric matrix,  it yields $L_aL_c=L_cL_a$ by \cite[p.\,233]{Horn:13}.
Moreover, using \cite[p.\,490]{Horn:13} leads to that   $L_aL_c$ is symmetric  positive definite.
\end{proof}

\begin{lemma}\label{lemma2.8}
Let $A^{(1)}:=A_h$ be defined by
(\ref{2.1}) and $A^{(k)}=I_k^{k-1}A^{(k-1)}I_{k-1}^{k}$. Then
\begin{equation*}
   \min_{\nu^{k-1} \in \mathcal{B}_{k-1} }||\nu^k-I_{k-1}^k\nu^{k-1}||_{A_k}^2\leq m_0 ||A_k\nu^k||_{D_k^{-1}}^2 \quad  \forall \nu^k \in \mathcal{B}_k
\end{equation*}
with $A_k=A^{(K-k+1)}$ and $m_0=(1+\widetilde{m}_0)^2$,
where
$$\widetilde{m}_0=\max\left\{\frac{\left(6C_k+2^{k-1}\right)\left(a_0+2a_1\right)}
{\left(2C_k+2^{k-1}\right)\left(a_0+2a_1\right)-2^{k+1}a_1} \quad\forall k \geq 1\right\}$$
and $C_k=2^{k-2}\cdot\frac{2^{2k-2}-1}{3}$.
In particular,
\begin{equation*}
 m_0=\left\{ \begin{split}
&1~~{\rm if }~~a_0+2a_1=0;\\
&16~~{\rm if}~~a_1\leq 0;\\
&\max\{25,{4a_0^2}/{\left(a_0-2a_1\right)^2}\}~~~{\rm if}~~a_1> 0, ~a_0 \neq 2a_1.
 \end{split}
 \right.
\end{equation*}
\end{lemma}

\begin{proof}
Let an odd number $M_k$ be defined  by (\ref{2.2}).
For any
$$\nu^k=(\nu^k_1,\nu^k_2,\ldots,\nu^k_{M_k})^{\rm T} \in \mathcal{B}_k~~~~{\rm and }~~~~\nu^k_0=\nu^k_{M_k+1}=0,$$
taking  $\nu^{k-1}=(\nu^k_2,\nu^k_4,\ldots,\nu^k_{M_k-1})^{\rm T} \in \mathcal{B}_{k-1}$ yields
\begin{equation*}
\begin{split}
\nu^{k-1}=T\nu^{k},
\end{split}
\end{equation*}
where the cutting matrix is defined by
\begin{equation}\label{2.15}
T=\left [ \begin{matrix}
0       &   1        & 0            &  0        &\cdots        &0       &0       &0       &0         \\
0       &   0        & 0            &  1        &\cdots        &0       &0       &0       &0          \\
\vdots  &\vdots      &\vdots        &\vdots     &\cdots        &\vdots  &\vdots  &\vdots  &\vdots    \\
0       &   0        & 0            &  0        &\cdots        &1       &0       &0       &0         \\
0       &   0        & 0            &  0        &\cdots        &0       &0       &1       &0
 \end{matrix}
 \right ]_{M_{k-1}\times M_{k}}.
\end{equation}
Therefore, we have
\begin{equation}\label{2.16}
\begin{split}
&\nu^k-I_{k-1}^k\nu^{k-1}=\left(I-I_{k-1}^kT\right)\nu^{k}\\
&=\left(\nu^k_1-\frac{\nu^k_0+\nu^k_2}{2},0,\nu^k_3-\frac{\nu^k_2+\nu^k_4}{2},0,
\ldots,\nu^k_{M_k}-\frac{\nu^k_{M_k-1}+\nu^k_{M_k+1}}{2}\right)^{\rm T}.
\end{split}
\end{equation}

Let $L_{M_k}={\rm tridiag}(-1, 2, -1)$ be the $M_k\times M_k$ one dimensional discrete Laplacian.
According to (\ref{2.9}) and (\ref{2.16}), there exists
\begin{equation}\label{2.17}
\begin{split}
||\nu^k-I_{k-1}^k\nu^{k-1}||_{L_{M_k}}^2
&=2||\nu^k-I_{k-1}^k\nu^{k-1}||^2\\
&=2||\left(I-I_{k-1}^kT\right)\nu^{k}||^2 \leq \frac{1}{2}||L_{M_k}\nu^k||^2,
\end{split}
\end{equation}
since
$$ 2 \sum_{i=1}^{(M_k+1)/2}\left( \nu^k_{2i-1}-\frac{\nu^k_{2i-2} +\nu^k_{2i}}{2}\right)^2
\leq   2 \sum_{i=1}^{M_k}\left( \nu^k_{i}-\frac{\nu^k_{i-1} +\nu^k_{i+1}}{2}\right)^2=\frac{1}{2}||L_{M_k}\nu^k||^2.
$$
From (\ref{2.16}), (\ref{2.17}) and (\ref{2.14}), we get
\begin{equation}\label{2.18}
\begin{split}
&||\nu^{K-k+1}-I_{K-k}^{K-k+1}\nu^{K-k}||_{A^{(k)}}^2\\
&=\left(2\mu_1+2\mu_2\right)||\nu^{K-k+1}-I_{K-k}^{K-k+1}\nu^{K-k}||^2\\
& \leq \frac{\mu_1+\mu_2}{2}||L_{M_{K-k+1}}\nu^{K-k+1}||^2.
\end{split}
\end{equation}
According to  Lemma \ref{lemma2.7} and (\ref{2.14}), it implies that $A_1^{(k)}A_2^{(k)}$ is symmetric positive definite. Therefore,
\begin{equation*}
\begin{split}
  ||A^{(k)}\nu^{K-k+1}||^2
  &\geq ||A_1^{(k)}\nu^{K-k+1}||^2=\mu_1^2||L_{M_{K-k+1}}\nu^{K-k+1}||^2,
\end{split}
\end{equation*}
which yields
\begin{equation}\label{2.19}
\begin{split}
  ||A^{(k)}\nu^{K-k+1}||_{(D^k)^{-1}}^2
  \geq \frac{\mu_1^2}{2\mu_1+2\mu_2}||L_{M_{K-k+1}}\nu^{K-k+1}||^2,
\end{split}
\end{equation}
where $D^k$ is the diagonal of $A^{(k)}$.
Using (\ref{2.18}) and (\ref{2.19}), there exists
\begin{equation*}
\begin{split}
||\nu^{K-k+1}-I_{K-k}^{K-k+1}\nu^{K-k}||_{A^{(k)}}^2
& \leq\frac{\mu_1+\mu_2}{2}||L_{M_{K-k+1}}\nu^{K-k+1}||^2\\
&\leq\left(1+\frac{\mu_2}{\mu_1}\right)^2 ||A^{(k)}\nu^{K-k+1}||_{(D^k)^{-1}}^2\\
\end{split}
\end{equation*}
with
\begin{equation}\label{2.20}
  \frac{\mu_2}{\mu_1}=
\frac{\left(6C_k+2^{k-1}\right)\left(a_0+2a_1\right)}
{\left(2C_k+2^{k-1}\right)\left(a_0+2a_1\right)-2^{k+1}a_1}\geq 0\quad{\rm for ~} k \geq 1;
\end{equation}
when $k=1$, it can be simplified as
\begin{equation}\label{2.21}
\frac{\mu_2}{\mu_1}=\frac{a_0+2a_1}{a_0-2a_1}.
\end{equation}
In particular, there exists
\begin{equation}\label{2.22}
  \frac{\mu_2}{\mu_1}\left\{ \begin{split}
&=0~~{\rm if }~~a_0+2a_1=0  \quad\forall k \geq 1;\\
&<3~~{\rm if }~~a_1\leq 0  \quad\forall k \geq 1;\\
&\leq 4~~{\rm if }~~a_1> 0  \quad\forall k \geq 2,
 \end{split}
 \right.
\end{equation}
since
\begin{equation*}
\begin{split}
  \frac{\mu_2}{\mu_1}
&=3+\frac{-2^{k}a_0+4\cdot 2^k a_1}{\left(2C_k-2^{k-1}\right)\left(a_0+2a_1\right)+2^{k}a_0}\\
&\leq 3+\frac{2^{k}a_0}{\left(2C_k-2^{k-1}\right)\left(a_0+2a_1\right)+2^{k}a_0}\leq 4~~{\rm with}~~a_1> 0 \quad \forall k \geq 2.\\
\end{split}
\end{equation*}
Hence
\begin{equation*}
   \min_{\nu^{k-1} \in \mathcal{B}_{k-1} }||\nu^k-I_{k-1}^k\nu^{k-1}||_{A_k}^2\leq m_0 ||A_k\nu^k||_{D_k^{-1}}^2 \quad  \forall \nu^k \in \mathcal{B}_k,
\end{equation*}
where
$m_0=(1+\widetilde{m}_0)^2$
with
\begin{equation}\label{2.23}
\widetilde{m}_0=\max\left\{\frac{\left(6C_k+2^{k-1}\right)\left(a_0+2a_1\right)}
{\left(2C_k+2^{k-1}\right)\left(a_0+2a_1\right)-2^{k+1}a_1} \quad\forall k \geq 1\right\}
\end{equation}
and $C_k=2^{k-2}\cdot\frac{2^{2k-2}-1}{3}$. In particular,  from (\ref{2.20})-(\ref{2.23}), there exists
$$
 m_0   =(1+\widetilde{m}_0)^2 =\left\{ \begin{split}
&1~~~{\rm if}~~a_0+2a_1=0;\\
&16~~~{\rm if}~~a_1\leq 0;\\
&\max\{25,{4a_0^2}/{\left(a_0-2a_1\right)^2}\}~~~{\rm if}~~a_1> 0, ~a_0 \neq 2a_1,
 \end{split}
 \right.
$$
  where we use $$\left(1+\frac{a_0+2a_1}{a_0-2a_1}\right)^2={4a_0^2}/{\left(a_0-2a_1\right)^2}.$$
The proof is completed.
\end{proof}
\begin{remark}
When $a_0=2a_1$, according to the theory of Toeplitz matrices generated by a function \cite{szego},  the generation function of
the considered tridiagonal Toeplitz matrices is $f(\theta)=2a_1(1+\cos\theta)$ and in that case the symbol has a zero at $\theta=\pi$: following the results in \cite{Fiorentino:91}[page 292, eq. (10), and Section 2.2.3], necessarily the symbol associated with the prolongation/restriction operator has to show a zero at $0$ and has to positive at $\pi$. This shows that the considered operators with stencil $[1~~ 2~~ 1]$ cannot be used in agreement with the considered condition, but the only possible tridiagonal choice is $[-1~~ 2~-1]$.
In fact,  we  know that the condition in Lemma \ref{lemma2.8}  is not only sufficient for optimality as shown here, but it is also necessary (see  \cite{Fiorentino:91,Fiorentino:96,Arico:04,Arico:07}).

The same type of connection is observed for the 2D case developed in Section \ref{sec3}.
\end{remark}

According to Lemmas \ref{lemma2.6},  \ref{lemma2.8} and Theorem \ref{theorem2.3}, we obtain the following result.
\begin{theorem}\label{theorem2.9}
For the algebraic system (\ref{2.1}),  we find
$$||I-B_kA_k||_{A_k} \leq \frac{m_0}{2l\omega+m_0}<1~~{\rm with}~~~1\leq k\leq K, ~~~~\omega \in (0,1/2],$$
where  the operator $B_k$ is defined by the  V-cycle method in  Multigrid Algorithm  \ref{MGM}  and  $l$ is the number of smoothing steps
and $m_0$ is given in Lemma \ref{lemma2.8}.
\end{theorem}

\section{Uniform convergence  of V-Cycle MGM  for 2D}\label{sec3}
In this section, we consider the symmetric positive definite Toeplitz block tridiagonal matrix. As an interesting example,
we study the algebraic system
\begin{equation}\label{3.1}
  {\bf A}_h{\bf v}^h={\bf f}_h,
\end{equation}
where $${\bf A}_h=c_1I \otimes I + c_2\left(I \otimes L +L \otimes I\right), ~~c_1\geq 0,~~c_2>0,$$
and $I$ is identity matrix, $L={\rm tridiag}(-1, 2, -1)$.
This example arises, for instance, from the discretization of the Poisson equations ($c_1=0$) in a square or the heat equations or the time fractional PDEs \cite{Deng:14,Fiorentino:96,Horton:95,Meurant:92}.

In 2D, the notations can be defined in a straightforward manner from the 1D case.
Let $\Omega \in (0,b)\times (0,b)$ and the mesh points $x_i=ih$, $y_j=jh$, $h=b/(M+1)$.
We still use the notation that $\mathcal{B}_k$ represents not only the grid with grid spacing $h_k=2^{(K-k)}h$, but also the space of vectors defined on that grid, where
\begin{equation}\label{3.2}
  \mathcal{B}_k=\Big\{(x_i^k,y_j^k)\Big|x_i^k=\frac{i}{2^k}b, y_j^k=\frac{j}{2^k}b, i,j=1: M_k \Big\}
\end{equation}
with $M_k=2^k-1, \,k=1 : K$.

For the two dimensional  system, the  restriction operator ${\bf I}_k^{k-1}$ and prolongation operator ${\bf I}_{k-1}^k$ \cite[p.\,436-439]{Saad:03} are, respectively, defined by
\begin{equation}\label{3.3}
 {\bf I}_{k-1}^k=P\otimes P:= I_{k-1}^k\otimes I_{k-1}^k,
\end{equation}
where $I_{k-1}^k$ is defined by (\ref{2.5}),  and
\begin{equation*}
\begin{split}
 {\bf I}_{k-1}^{k} =4\left( {\bf I}_k^{k-1}\right)^T.
\end{split}
\end{equation*}
The coarse problem is typically defined by the Galerkin approach
\begin{equation}\label{3.4}
 {\bf A}_{k-1}={\bf I}_k^{k-1}{\bf A}_k{\bf }{\bf I}_{k-1}^{k},~~~~~~{\bf f}^{k-1}={\bf I}_k^{k-1}{\bf f}^k.
\end{equation}

Let ${\bf K}_k$ be the iteration matrix of the smoothing operator. In this work, we
 take ${\bf K}_k$ to be  the weighted (damped) Jacobi iteration matrix
\begin{equation}\label{3.5}
  {\bf K}_k={\bf I}-{\bf S}_k{\bf A}_k,~~{\rm where}~~{\bf S}_{k}:={\bf S}_{k,\omega}=\omega {\bf D}_k^{-1}
\end{equation}
with a weighting  factor $\omega \in (0,1/4]$, and ${\bf D}_k$ is the diagonal of ${\bf A}_k$.

\subsection{Convergence estimates  of MGM for  2D}
We now give a complete proof on the uniform convergence  of the MGM for the algebraic system (\ref{3.1}),
i.e., we need to examine  the two assumptions (\ref{2.10}) and (\ref{2.12}). First, we give some lemmas.
\begin{lemma}\label{lemma3.1}\cite[p.\,5]{Chan:07}
Let  $A$ be a   symmetric matrices. Then
$$
\lambda_{\min}(A)=\min_{x\neq 0}\frac{x^TAx}{x^Tx},~~ \lambda_{\max}(A)=\max_{x\neq 0}\frac{x^TAx}{x^Tx}.
$$
\end{lemma}
\begin{lemma}\cite[p.\,27]{Quarteroni:07}\label{lemma3.2}
The matrix $A \in  {C}^{n\times n}$ is positive definite if and only if it is hermitian and has positive eigenvalues.
\end{lemma}

\begin{lemma}\cite[p.\,140]{Laub:05}\label{lemma3.3}
Let $A \in  \mathbb{{R}}^{m\times n}$, $B \in  \mathbb{{R}}^{r\times s}$, $C \in  \mathbb{{R}}^{n\times p}$, and $D \in  \mathbb{{R}}^{s\times t}$.
Then
\begin{equation*}
  (A \otimes B)(C \otimes D)=AC \otimes  BD \quad (\in  \mathbb{{R}}^{mr\times pt}).
\end{equation*}
Moreover, for all $A$ and $B$, $(A \otimes B)^T=A^T\otimes B^T$.
\end{lemma}

\begin{lemma}\cite[p.\,141]{Laub:05}\label{lemma3.4}
Let $A \in  \mathbb{{R}}^{n\times n}$ and $\{\lambda_i\}_{i=1}^n$ be its eigenvalues; let $B \in  \mathbb{{R}}^{m\times m}$ and $\{\mu_j\}_{j=1}^m$ be its  eigenvalues .
Then the $mn$ eigenvalues of $A \otimes B$ are
\begin{equation*}
  \lambda_1\mu_1,\ldots,\lambda_1\mu_m, \lambda_2\mu_1,\ldots,\lambda_2\mu_m,\ldots,\lambda_n\mu_1\ldots,\lambda_n\mu_m.
\end{equation*}
\end{lemma}
\begin{lemma}\cite[p.\,396]{Golub:96}\label{lemma3.5}
If $P$ and $P+Q$ are n-by-n symmetric matrices, then
\begin{equation*}
  \lambda_k(P)+ \lambda_1(Q) \leq \lambda_k(P+Q) \leq \lambda_k(P) + \lambda_n(Q), \quad k=1,2,\ldots,n.
\end{equation*}
\end{lemma}
\begin{lemma}\label{lemma3.6}
Let ${\bf A}^{(1)}:={\bf A}_h$ be defined by
(\ref{3.1}) and ${\bf A}^{(k)}={\bf I}_k^{k-1}{\bf A}^{(k-1)}{\bf I}_{k-1}^{k}$. Then
\begin{equation*}
\frac{\omega}{\lambda_{\max}({\bf A}_k) }({\bf v}^k,{\bf v}^k) \leq ({\bf S}_k{\bf v}^k,{\bf v}^k)\leq ({\bf A}_k^{-1}{\bf v}^k,{\bf v}^k)
\quad \forall {\bf v}^k\in \mathcal{B}_k,
\end{equation*}
where ${\bf A}_k={\bf A}^{(K-k+1)}$, ${\bf S}_k=\omega {\bf D}_k^{-1}$, $\omega \in (0,1/4]$ and ${\bf D}_k$ is the diagonal of ${\bf A}_k$.
\end{lemma}
\begin{proof}
Given a sequence ${\bf Z}^{(k)}$ and $Z^{(k)}$, $k\geq 1$, we denote
\begin{equation}\label{3.6}
{\bf Z}^{(k)}={\bf I}_k^{k-1}{\bf Z}^{(k-1)}{\bf I}_{k-1}^{k}~~{\rm and } ~~ Z^{(k)}= I_k^{k-1} Z^{(k-1)} I_{k-1}^{k}.
\end{equation}
In the following, ${\bf Z}$ and $Z$ given in (\ref{3.6}) can also be taken as ${\bf A}$ and $M$, etc.

Taking the block matrix
\begin{equation}\label{3.7}
  {\bf Z}^{(1)}=M^{(1)}\otimes N^{(1)},
\end{equation}
there exists
\begin{equation}\label{3.8}
  {\bf Z}^{(k)}= \left(I_k^{k-1} M^{(k-1)} I_{k-1}^{k} \right)\otimes \left(I_k^{k-1} N^{(k-1)} I_{k-1}^{k}\right)=M^{(k)}\otimes N^{(k)}.
\end{equation}
Combining (\ref{3.6})-(\ref{3.8}) and ${\bf A}^{(1)}=c_1I \otimes I + c_2\left(I \otimes L +L \otimes I\right)$, we obtain
\begin{equation}\label{3.9}
\begin{split}
 {\bf A}^{(k)}
=&c_1I^{(k)}\otimes I^{(k)}+c_2\left(I^{(k)}\otimes L^{(k)}+ L^{(k)} \otimes I^{(k)}\right).
\end{split}
\end{equation}
According to Corollary \ref{corollary2.5} and (\ref{2.14}), we have
\begin{equation}\label{3.10}
\begin{split}
  I^{(k)}&=\frac{1}{8^{k-1}}{\rm tridiag}(C_k, 4C_k+2^{k-1}, C_k)=\theta_1 I+\theta_3\widetilde{L};\\
  L^{(k)}&=\frac{1}{8^{k-1}}{\rm tridiag}(-2^{k-1}, 2^{k}, -2^{k-1})=\theta_2L,
\end{split}
\end{equation}
where $\widetilde{L}={\rm tridiag}(1, 2, 1)$, $L={\rm tridiag}(-1, 2, -1)$ and
\begin{equation}\label{3.11}
  \theta_1=\frac{2C_k+2^{k-1}}{ 8^{k-1}}>0,~~~~\theta_2=\frac{2^{k-1}}{ 8^{k-1}}>0,~~~~\theta_3=\frac{C_k}{ 8^{k-1}}\geq 0.
\end{equation}
Next we prove
$$1\leq \lambda_{\max}\left(\left({\bf D}^{(k)}\right)^{-1}{\bf A}^{(k)}\right)<4.$$
The maximum eigenvalues of   $I^{(k)}$ and $L^{(k)}$ are, respectively, given by \cite[p.\,702]{Stoer:02}
\begin{equation*}
  \lambda_{\max}\left(I^{(k)}\right)<\frac{ 6C_k+2^{k-1}}{8^{k-1}}=3\theta_1-2\theta_2,~~~~
  \lambda_{\max}\left(L^{(k)}\right)<\frac{ 2^{k+1}}{8^{k-1}}=4\theta_2.
\end{equation*}
Using Lemmas \ref{lemma3.3}-\ref{lemma3.5} and (\ref{3.9}), we obtain
\begin{equation}\label{3.12}
  \lambda_{\max}\left({\bf A}^{(k)}\right)<\eta_1~~~~{\rm with}~~~~\eta_1=c_1\left(3\theta_1-2\theta_2\right)^2
  +8c_2\left(3\theta_1-2\theta_2\right) \theta_2,
\end{equation}
and
\begin{equation}\label{3.13}
\begin{split}
\eta_2:=\lambda\left({\bf D}^{(k)}\right)
&=  \lambda_{\max}\left({\bf D}^{(k)}\right)=\lambda_{\min}\left({\bf D}^{(k)}\right)\\
&=c_1\left(\frac{ 4C_k+2^{k-1}}{8^{k-1}}\right)^2
  +2c_2\frac{ 2^{k}}{8^{k-1}}\cdot\frac{ 4C_k+2^{k-1}}{8^{k-1}}\\
&=c_1\left(2\theta_1-\theta_2\right)^2
  +4c_2\left(2\theta_1-\theta_2\right)\theta_2,
\end{split}
\end{equation}
which yields
\begin{equation}\label{3.14}
  \lambda_{\max}\left(\left({\bf D}^{(k)}\right)^{-1}{\bf A}^{(k)}\right)<\frac{\eta_1}{\eta_2}<4.
\end{equation}
If we take  ${\bf x}=[1,0,\ldots,0]^T$, then
$$\lambda_{\max}({\bf A}^{(k)})\geq \frac{{\bf x}^T{\bf A}^{(k)}{\bf x}}{{\bf x}^T{\bf x}}=\lambda_{\max}\left({\bf D}^{(k)}\right).$$
The proof is completed.
\end{proof}

\begin{lemma}\label{lemma3.7}
Let ${\bf A}^{(1)}:={\bf A}_h$ be defined by
(\ref{3.1}) and ${\bf A}^{(k)}={\bf I}_k^{k-1}{\bf A}^{(k-1)}{\bf I}_{k-1}^{k}$. Then
\begin{equation*}
   \min_{{\bf v}^{k-1} \in \mathcal{B}_{k-1} }||{\bf v}^k-{\bf I}_{k-1}^k{\bf v}^{k-1}||_{{\bf A}_k}^2
   \leq {\bf m}_0 ||{\bf A}_k{\bf v}^k||_{{\bf D}_k^{-1}}^2 \quad  \forall {\bf v}^k \in \mathcal{B}_k
\end{equation*}
with ${\bf A}_k={\bf A}^{(K-k+1)}$ and ${\bf m}_0=1536<\infty$.
\end{lemma}
\begin{proof}
Let an odd number $M_k$ be defined  by (\ref{3.2}).
For any
\begin{equation*}
  {\bf v}^k=({\bf v}^k_1,{\bf v}^k_2,\ldots,{\bf v}^k_{M_k})^{\rm T} \in \mathcal{B}_k~~~~{\rm and }~~~~{\bf v}^k_0={\bf v}^k_{M_k+1}=0
\end{equation*}
with
${\bf v}^k_i=({\bf v}^k_{i,1},{\bf v}^k_{i,2},\ldots,{\bf v}^k_{i,{M_k}})^{\rm T},$
and taking
\begin{equation*}
{\bf v}^{k-1}=({\bf \widetilde{v}}^k_2,{\bf \widetilde{v}}^k_4,\ldots,{\bf \widetilde{v}}^k_{M_k-1})^{\rm T} \in \mathcal{B}_{k-1},
\end{equation*}
with ${\bf \widetilde{v}}^k_i=({\bf v}^k_{i,2},{\bf v}^k_{i,4},\ldots,{\bf v}^k_{i,{M_k-1}})^{\rm T},$
there exists
\begin{equation*}
\begin{split}
{\bf \widetilde{v}}_i^{k}=T{\bf v}^k_i,
\end{split}
\end{equation*}
where the cutting  matrix $T$ is defined by (\ref{2.15}).
Using the  above equations, it yields
\begin{equation}\label{3.15}
  {\bf v}^{k-1}=\left(T\otimes T\right) {\bf v}^{k}.
\end{equation}
From (\ref{3.3}) and (\ref{3.15}), we get
\begin{equation}\label{3.16}
 {\bf I}_{k-1}^k{\bf v}^{k-1}=\left(PT\otimes PT\right) {\bf v}^{k}.
\end{equation}
Thus
\begin{equation}\label{3.17}
\begin{split}
 &{\bf v}^{k}- {\bf I}_{k-1}^k{\bf v}^{k-1}\\
&=\left(I\otimes I-PT\otimes PT\right) {\bf v}^{k},\\
&=\Big({\bf v}^k_1-\frac{PT}{2}\left({\bf v}^k_0+{\bf v}^k_2\right),\left(I-PT\right){\bf v}^k_2,
{\bf v}^k_3-\frac{PT}{2}\left({\bf v}^k_2+{\bf v}^k_4\right), \left(I-PT\right){\bf v}^k_4,\\
&\qquad \ldots,\left(I-PT\right){\bf v}^k_{M_k-1}, {\bf v}^k_{M_k}-\frac{PT}{2}\left({\bf v}^k_{M_k-1}+{\bf v}^k_{M_k+1}\right) \Big)^{\rm T}.
\end{split}
\end{equation}
Hence, we obtain
\begin{equation}\label{3.18}
\begin{split}
 &\|{\bf v}^{k}- {\bf I}_{k-1}^k{\bf v}^{k-1}\|^2\\
 &=\sum_{i=1}^{(M_k+1)/2}\Big|\Big| {\bf v}^k_{2i-1}-\frac{PT}{2}\left({\bf v}^k_{2i-2} +{\bf v}^k_{2i}\right)\Big|\Big|^2
   +\sum_{i=1}^{(M_k+1)/2}\Big|\Big| \left(I-PT\right){\bf v}^k_{2i}\Big|\Big|^2\\
&\leq 2\sum_{i=1}^{(M_k+1)/2}\Big|\Big| {\bf v}^k_{2i-1}-\frac{{\bf v}^k_{2i-2} +{\bf v}^k_{2i}}{2}\Big|\Big|^2
   +3\sum_{i=1}^{(M_k+1)/2}\Big|\Big| \left(I-PT\right){\bf v}^k_{2i}\Big|\Big|^2,
\end{split}
\end{equation}
where we use
\begin{equation*}
\begin{split}
&\Big|\Big| {\bf v}^k_{2i-1}-\frac{PT}{2}\left({\bf v}^k_{2i-2} +{\bf v}^k_{2i}\right)\Big|\Big|^2\\
&\leq 2\Big|\Big| {\bf v}^k_{2i-1}-\frac{{\bf v}^k_{2i-2} +{\bf v}^k_{2i}}{2}\Big|\Big|^2
+\Big|\Big| \left(I-PT\right){\bf v}^k_{2i-2}\Big|\Big|^2+\Big|\Big| \left(I-PT\right){\bf v}^k_{2i}\Big|\Big|^2.
\end{split}
\end{equation*}
From (\ref{2.17}), we have
$||\left(I-PT\right){\bf v}^k_{2i}||^2 \leq \frac{1}{4}||L{\bf v}^k_{2i}||^2,$
which yields
\begin{equation}\label{3.19}
\begin{split}
\sum_{i=1}^{(M_k+1)/2}\Big|\Big| \left(I-PT\right){\bf v}^k_{2i}\Big|\Big|^2
&\leq \frac{1}{4} \sum_{i=1}^{(M_k+1)/2}||L{\bf v}^k_{2i}||^2\\
&\leq \frac{1}{4} \sum_{i=1}^{M_k}||L{\bf v}^k_{i}||^2
= \frac{1}{4} \left(\left(I\otimes L^2\right) {\bf v}^{k},{\bf v}^{k}\right),
\end{split}
\end{equation}
and
\begin{equation}\label{3.20}
\begin{split}
\sum_{i=1}^{(M_k+1)/2}\Big|\Big| {\bf v}^k_{2i-1}-\frac{{\bf v}^k_{2i-2} +{\bf v}^k_{2i}}{2}\Big|\Big|^2
&\leq \sum_{i=1}^{M_k}\Big|\Big| {\bf v}^k_{i}-\frac{{\bf v}^k_{i-1} +{\bf v}^k_{i+1}}{2}\Big|\Big|^2\\
&=\frac{1}{4}\left(\left(L^2\otimes I \right) {\bf v}^{k},{\bf v}^{k}\right).
\end{split}
\end{equation}
According to  (\ref{3.12}) and (\ref{3.18})-(\ref{3.20}), there exists
\begin{equation}\label{3.21}
\begin{split}
&||{\bf v}^{K-k+1}-{\bf I}_{K-k}^{K-k+1}{\bf v}^{K-k}||_{{\bf A}^{(k)}}^2\\
& \leq \lambda_{\max}\left({\bf A}^{(k)}\right) ||{\bf v}^{K-k+1}-{\bf I}_{K-k}^{K-k+1}{\bf v}^{K-k}||^2\\
& \leq \frac{3\eta_1}{4} \left(\left( I\otimes L^2 +  L^2\otimes I\right) {\bf v}^{K-k+1},{\bf v}^{K-k+1}\right).
\end{split}
\end{equation}

From Lemmas \ref{lemma3.2}-\ref{lemma3.4} and  \ref{lemma2.7}, we know that
the matrix $AC \otimes  BD $ is symmetric positive definite, where $A$ (or $B$, $C$,$D$) can be chosen as $I$ (or $L$, $\widetilde{L}$).
Thus using (\ref{3.9}) and (\ref{3.10}), there exists
\begin{equation}\label{3.22}
\begin{split}
||{\bf A}^{(k)}{\bf v}^{K-k+1}||^2
&\geq ||{\bf A}_1^{(k)}{\bf v}^{K-k+1}||^2
\geq \left({\bf B}^{(k)}{\bf v}^{K-k+1},{\bf v}^{K-k+1}\right)\\
&\geq \left({\bf C}^{(k)}{\bf v}^{K-k+1},{\bf v}^{K-k+1}\right),
\end{split}
\end{equation}
where
\begin{equation}\label{3.23}
\begin{split}
{\bf A}_1^{(k)}&=c_1\theta_1^2I\otimes I +c_2\theta_1\theta_2\left(I\otimes L+L\otimes I\right),\\
{\bf B}^{(k)}&=c_1^2\theta_1^4I\otimes I +2c_1c_2\theta_1^3\theta_2\left(I\otimes L+L\otimes I\right)
+c_2^2\theta_1^2\theta_2^2\left(I\otimes L^2+L^2\otimes I\right),\\
{\bf C}^{(k)}&=\eta_3\left(I\otimes L^2+L^2\otimes I\right)~~~~{\rm with}~~~~
\eta_3=\frac{c_1^2\theta_1^4}{32} +\frac{c_1c_2\theta_1^3\theta_2}{2}+c_2^2\theta_1^2\theta_2^2.
\end{split}
\end{equation}
Combining (\ref{3.13}) and (\ref{3.21})-(\ref{3.23}), we have
\begin{equation*}
\begin{split}
  ||{\bf A}^{(k)}{\bf v}^{K-k+1}||_{({\bf D}^k)^{-1}}^2
 & =\frac{1}{\eta_2} ||{\bf A}^{(k)}{\bf v}^{K-k+1}||^2\\
 & \geq \frac{\eta_3}{\eta_2} \left(\left(I\otimes L^2+L^2\otimes I\right){\bf v}^{K-k+1},{\bf v}^{K-k+1}\right)\\
 & \geq \frac{4\eta_3}{3\eta_1\eta_2}||{\bf v}^{K-k+1}-{\bf I}_{K-k}^{K-k+1}{\bf v}^{K-k}||_{{\bf A}^{(k)}}^2.
\end{split}
\end{equation*}
According to  (\ref{3.14}), (\ref{3.13}) and (\ref{3.23}), there exists
\begin{equation*}
\begin{split}
\frac{4\eta_3}{3\eta_1\eta_2}>\frac{\eta_3}{3\eta_2^2}
>\frac{\eta_3}{48\left(c_1\theta_1^2+2c_2\theta_1\theta_2 \right)^2}
>\frac{1}{1536}>0, \quad\forall k \geq 1.
\end{split}
\end{equation*}
 More concretely, from  (\ref{3.14}) we get $\frac{4\eta_3}{3\eta_1\eta_2}>\frac{\eta_3}{3\eta_2^2}$;  and using (\ref{3.13}) and (\ref{3.23}), there exists
\begin{equation*}
\begin{split}
\eta_2
=c_1\left(2\theta_1-\theta_2\right)^2+4c_2\left(2\theta_1-\theta_2\right)\theta_2
\leq c_1\left(2\theta_1\right)^2+4c_22\theta_1\theta_2=4\left(c_1\theta_1^2+2c_2\theta_1\theta_2\right),
\end{split}
\end{equation*}
and
\begin{equation*}
\begin{split}
\frac{4\eta_3}{3\eta_1\eta_2}>\frac{\eta_3}{3\eta_2^2}
>\frac{\eta_3}{48\left(c_1\theta_1^2+2c_2\theta_1\theta_2 \right)^2}
>\frac{\left(c_1\theta_1^2+2c_2\theta_1\theta_2 \right)^2}{32*48\left(c_1\theta_1^2+2c_2\theta_1\theta_2 \right)^2}=\frac{1}{1536}.
\end{split}
\end{equation*}
Hence
\begin{equation*}
   \min_{{\bf v}^{k-1} \in \mathcal{B}_{k-1} }||{\bf v}^k-{\bf I}_{k-1}^k{\bf v}^{k-1}||_{{\bf A}_k}^2
   \leq 1536 ||{\bf A}_k{\bf v}^k||_{{\bf D}_k^{-1}}^2 \quad  \forall {\bf v}^k \in \mathcal{B}_k.
\end{equation*}
The proof is completed.
\end{proof}

Following the above results, we obtain the uniform convergence of the V-cycle Multigrid method.
\begin{theorem}\label{theorem3.8}
For the algebraic system (\ref{3.1}),  it satisfies
$$||{\bf I}-{\bf B}_k{\bf A}_k||_{{\bf A}_k} \leq \frac{{\bf m}_0}{2l\omega+{\bf m}_0}<1~~{\rm with}~~~1\leq k\leq K, ~~~~\omega \in (0,1/4],$$
where  the operator ${\bf B}_k$ is defined by the  V-cycle method in  Multigrid Algorithm  \ref{MGM}  and  $l$ is the number of smoothing steps
and ${\bf m}_0$ is given in Lemma \ref{lemma3.7}.
\end{theorem}
\begin{remark}
Based on the above analysis,  the  convergence estimates of MGM is easy to obtain for   the two-dimensional  compact difference  scheme  $  {\bf A}_h{\bf v}^h={\bf f}_h$, where
$
{\bf A}_h=c_1H\otimes H+  c_2\left( H\otimes L+L\otimes H \right),
$
and  the matrix ${H}=\frac{1}{12}{\rm tridiag}(1, 10, 1)$.
\end{remark}
\section{The finite difference scheme for Feynman-Kac  equation}
Let $T>0$,  $\Omega=(0,b)\times (0,b)$.
Without loss of generality, we add a force term $f({\bf x},\rho,t)$ on the right hand side of (\ref{1.2})
and make it subject to the given initial and boundary conditions, which leads to
\begin{equation}\label{4.1}
\begin{split}
{^s_c}{D}_t^\alpha G({\bf x},\rho,t)
&={^s\!}D_t^\alpha [G({\bf x},\rho,t)-e^{-\rho t}G({\bf x},\rho,0)]\\
&  = \kappa_\alpha \Delta G({\bf x},\rho,t)+f({\bf x},\rho,t),  ~~~~0<t \leq T,~~{\bf x} \in \Omega
\end{split}
\end{equation}
with  the initial and boundary conditions
\begin{equation*}
\begin{split}
&G({\bf x},\rho,0)=\phi({\bf x}),~~{\bf x}\in \Omega,\\
&G({\bf x},\rho,t)=\psi(t), ~~({\bf x},t) \in \partial \Omega \times [0,T].
\end{split}
\end{equation*}

\subsection{Derivation of the compact difference  scheme for 1D}
Let the mesh points
$$\Omega_h=\{x_i=ih|0\leq i\leq M+1\}~~ {\rm and}~~ \Omega_\tau=\{t_n=n\tau|0\leq n\leq N\},$$
where $h=b/(M+1)$ and $\tau=T/N$ are the uniform space stepsize and time steplength, respectively.
Let  $\mathcal{V}=\{v_i^n| 0 \leq i \leq M+1, 0\leq n\leq N \}$ be the gird function defined on
the mesh $\Omega_h\times \Omega_\tau$. For any grid function $v_i^n \in \mathcal{V}$,  we denote
\begin{equation}\label{4.2}
  \delta_x^2v_i^n=\frac{1}{h^2}(v_{i-1}^n-2v_i^n+v_{i+1}^n),
\end{equation}
and the compact operator
\begin{equation}\label{4.3}
\mathcal{C}_hv_i^n=\left\{ \begin{array}
 {l@{\quad} l}
(1+\frac{h^2}{12}\delta_x^2)v_i^n=\frac{1}{12}(v_{i-1}^n+10v_i^n+v_{i+1}^n),~~~~1 \leq i \leq M,\\
\\
v_i^n, ~~~~i=0 ~{\rm or}~ M+1.
\end{array}
 \right.
\end{equation}
Then, we obtain the fourth-order accuracy compact operator in spatial direction; see the following lemma.
\begin{lemma}[\cite{Ji:15}]\label{lemmma4.1}
Let  $G(x)\in C^6(\Omega )$ and $\theta(s)=5(1-s)^3-3(1-s)^5$. Then
$$\mathcal{C}_h \left[\frac{\partial^2}{\partial x^2} G(x)\Big|_{x=x_i}\right]
=\delta_x^2G(x_i)+\frac{h^4}{360}\int_0^1\left[G^{(6)}(x_i-s h)+G^{(6)}(x_i+s h)\right]\theta(s)ds$$
with $x_i=ih$, $1 \leq i \leq {M}$.
\end{lemma}

Denote $G_{i,\rho}^n$ and $f_{i,\rho}^n$, respectively, as the numerical approximation to $G(x_i,\rho,t_n)$ and $f(x_i,\rho,t_n)$. In this paper, we restrict $U(x)=1$ appeared in (\ref{1.1}); for the discussions of the more general choices of $U(x)$, see \cite{Deng:14}.
Using  \cite{Chen:13},  we obtain the $\nu$-th order approximations for the Riemann-Liouville fractional substantial derivative, i.e.,
\begin{equation}\label{4.4}
\begin{split}
&{^s\!}D_t^\alpha G(x,\rho,t)|_{t=t_n}=\frac{1}{\tau^\alpha}\sum_{k=0}^{n}{d}_{k}^{\nu,\alpha}G(x,\rho,t_{n-k})+\mathcal{O}(\tau^\nu);\\
& {^s\!}D_t^\alpha [e^{-\rho t}G(x,\rho,0)]_{t=t_n}=\frac{1}{\tau^\alpha}\sum_{k=0}^{n}d_{k}^{\nu,\alpha}e^{-\rho (n-k)\tau}G(x,\rho,0)+\mathcal{O}(\tau^\nu)
\end{split}
\end{equation}
with
\begin{equation}\label{4.5}
d_{k}^{\nu,\alpha}=e^{-\rho  k \tau}{l}_k^{\nu,\alpha},~~\nu=1,2,3,4,
\end{equation}
where ${l}_k^{1,\alpha}$, ${l}_k^{2,\alpha}$, ${l}_k^{3,\alpha}$ and ${l}_k^{4,\alpha}$
are given in \cite{Chen:0013,Chen:1313}.
In particular, when $\nu=1$, there exists
\begin{equation}\label{4.6}
  d_{k}^{1,\alpha}=e^{-\rho k \tau}{l}_k^{1,\alpha},~~ {l}_k^{1,\alpha}=(-1)^k\left ( \begin{matrix} \alpha \\ k\end{matrix} \right ).
\end{equation}
From  (\ref{4.4}) and (\ref{4.5}), there exists $\nu$-th order approximations for Caputo fractional substantial derivative
\begin{equation}\label{4.7}
\begin{split}
&{^s_c}{D}_t^\alpha G(x_i,\rho,t)|_{t=t_n}\\
&=\frac{1}{\tau^\alpha}\sum_{k=0}^{n}{d}_{k}^{\nu,\alpha}\left[G(x_i,\rho,t_{n-k})-e^{-\rho (n-k)\tau}G(x_i,\rho,0)\right]+r_i^n
\end{split}
\end{equation}
with $|r_i^n|=\mathcal{O}(\tau^\nu)$, $\nu=1,2,3,4$.

Performing both sides of (\ref{4.1}) by $\mathcal{C}_h$ at the point $(x_i,t_n)$ results in
\begin{equation}\label{4.8}
\begin{split}
\mathcal{C}_h\left[{^s_c}{D}_t^\alpha G(x_i,\rho,t_n)\right]
  = \kappa_\alpha \mathcal{C}_h\left[\frac{\partial^2}{\partial x^2} G(x_i,\rho,t_n)\right]+\mathcal{C}_h\left[f(x_i,\rho,t_n)\right].
\end{split}
\end{equation}
According to (\ref{4.8}), (\ref{4.7}) and Lemma \ref{lemmma4.1}, Eq. (\ref{4.1}) can be rewritten as
\begin{equation}\label{4.9}
\begin{split}
&\mathcal{C}_h\left[\frac{1}{\tau^\alpha}\sum_{k=0}^{n}{d}_{k}^{\nu,\alpha}\left[G(x_i,\rho,t_{n-k})-e^{-\rho (n-k)\tau}G(x_i,\rho,0)\right]\right]\\
& = \kappa_\alpha \delta_x^2 G(x_i,\rho,t_n)+\mathcal{C}_h\left[f(x_i,\rho,t_n)\right]+\widetilde{r}_i^n
\end{split}
\end{equation}
with the local truncation error
\begin{equation*}
\begin{split}\widetilde{r}_i^n
&=\kappa_\alpha\frac{h^4}{360}\int_0^1\left[G^{(6)}(x_i-s h)+G^{(6)}(x_i+s h)\right]\theta(s)ds-\mathcal{C}_hr_i^n\\
&=\mathcal{O}(\tau^\nu+h^4),~ \nu=1,2,3,4,
\end{split}
\end{equation*}
i.e.,
\begin{equation}\label{4.10}
\begin{split}
|\widetilde{r}_i^n|\leq C_G  (\tau^\nu+h^4),~ \nu=1,2,3,4,
\end{split}
\end{equation}
where $C_G$ is a constant independent of $\tau$ and $h$.

Multiplying  (\ref{4.9}) by $\tau^\alpha$ leads to
\begin{equation}\label{4.11}
\begin{split}
&\mathcal{C}_h\left[\sum_{k=0}^{n}{d}_{k}^{\nu,\alpha}\left[G(x_i,\rho,t_{n-k})-e^{-\rho (n-k)\tau}G(x_i,\rho,0)\right]\right]\\
& = \kappa_\alpha \tau^\alpha\delta_x^2 G(x_i,\rho,t_n)+\tau^\alpha\mathcal{C}_h\left[f(x_i,\rho,t_n)\right]+R_i^n
\end{split}
\end{equation}
with
\begin{equation}\label{4.12}
\begin{split}
|R_i^n|=|\tau^\alpha \widetilde{r}_i^n|\leq C_G \tau^\alpha (\tau^\nu+h^4),~ \nu=1,2,3,4,
\end{split}
\end{equation}
where $C_G$ is given in (\ref{4.10}).

Using  (\ref{4.11}) and (\ref{4.5}) leads to the compact difference scheme  of (\ref{4.1}) as
\begin{equation}\label{4.13}
\begin{split}
&{l}_{0}^{\nu,\alpha}\frac{G_{i-1,\rho}^{n}+10G_{i,\rho}^{n}+G_{i+1,\rho}^{n}}{12}
+\mu_{h,\tau}^\alpha\left(- G_{i-1,\rho}^{n}+2G_{i,\rho}^{n}-G_{i+1,\rho}^{n}\right)\\
&=-\sum_{k=1}^{n-1}e^{-\rho  k \tau}{l}_k^{\nu,\alpha}
\frac{G_{i-1,\rho}^{n-k}+10G_{i,\rho}^{n-k}+G_{i+1,\rho}^{n-k}}{12}\\
& \quad                  +\sum_{k=0}^{n-1}e^{-\rho  n \tau}{l}_k^{\nu,\alpha}
\frac{G_{i-1,\rho}^{0}+10G_{i,\rho}^{0}+G_{i+1,\rho}^{0}}{12}\\
    & \quad                 +\tau^\alpha  \frac{f_{i-1,\rho}^{n}+10f_{i,\rho}^{n}+f_{i+1,\rho}^{n}}{12}
\end{split}
\end{equation}
with $\mu_{h,\tau}^\alpha =\kappa_\alpha \frac{\tau^\alpha}{h^2}$.
For the convenience of implementation, we use the matrix form of the grid functions
 \begin{equation*}
G^{n}=[G_{1,\rho}^n,G_{2,\rho}^n,\ldots,G_{M,\rho}^n]^{\rm T}
 ~~~~~~{\rm and}~~
F^{n}=[f_{1,\rho}^n,f_{2,\rho}^n,\ldots,f_{M,\rho}^n]^{\rm T}.
  \end{equation*}
Thus the compact difference scheme (\ref{4.13}) reduces to the following form
\begin{equation}\label{4.14}
\begin{split}
&{l}_{0}^{\nu,\alpha}{H}G^{n}+\mu_{h,\tau}^\alpha LG^{n}\\
&=-\sum_{k=1}^{n-1}e^{-\rho  k \tau}{l}_k^{\nu,\alpha}{H}G^{n-k}
               +\sum_{k=0}^{n-1}e^{-\rho  n \tau}{l}_k^{\nu,\alpha}{H}G^{0}
               +\tau^\alpha{H}F^{n}+\widetilde{F}^n.
\end{split}
\end{equation}
Here, the matrices ${H}=\frac{1}{12}{\rm tridiag}(1, 10, 1)$ and $L={\rm tridiag}(-1, 2, -1)$, i.e.,
\begin{equation}\label{4.15}
\begin{split}
{H}=\frac{1}{12}\left [ \begin{matrix}
10       &   1         &             &           \\
1        &   10        & 1           &            \\
         &\ddots       &\ddots       &\ddots       \\
         &             &1            &10
 \end{matrix}
 \right ]
 ~~~~{\rm and}~~~~
 L=\left [ \begin{matrix}
2        &   -1       &             &             \\
-1       &   2        & -1          &              \\
         &\ddots      &\ddots       &\ddots         \\
         &            &-1           &2
 \end{matrix}
 \right ],
\end{split}
\end{equation}
and $\widetilde{F}^n=[\widetilde{f}_{1,\rho}^n,0,\ldots,0,\widetilde{f}_{M,\rho}^n]^{\rm T}$
with the initial and boundary conditions
\begin{equation*}
\begin{split}
\widetilde{f}_{1,\rho}^n
=&\mu_{h,\tau}^\alpha G_{0,\rho}^n+\frac{1}{12}\Big[-l_{0}^{\nu,\alpha}G_{0,\rho}^n\\
&-\sum_{k=1}^{n-1}e^{-\rho  k \tau}{l}_k^{\nu,\alpha}G_{0,\rho}^{n-k}
+\sum_{k=0}^{n-1}e^{-\rho  n \tau}{l}_k^{\nu,\alpha}G_{0,\rho}^{0}+\tau^\alpha f_{0,\rho}^{n}\Big];
\end{split}
\end{equation*}
and
\begin{equation*}
\begin{split}
\widetilde{f}_{M,\rho}^n
=&\mu_{h,\tau}^\alpha G_{M+1,\rho}^n+\frac{1}{12}\Big[-l_{0}^{\nu,\alpha}G_{M+1,\rho}^n\\
&-\sum_{k=1}^{n-1}e^{-\rho  k \tau}{l}_k^{\nu,\alpha}G_{M+1,\rho}^{n-k}
+\sum_{k=0}^{n-1}e^{-\rho  n \tau}{l}_k^{\nu,\alpha}G_{M+1,\rho}^{0}+\tau^\alpha f_{M+1,\rho}^{n}\Big].
\end{split}
\end{equation*}

\subsection{Derivation of the center difference  scheme for 2D}
Let the mesh points $x_i=ih$, $y_j=jh$,
$t_n=n\tau$ with $0\leq i,j\leq M+1$, $0\leq n \leq {N}$, where $h=b/(M+1)$ and $\tau=T/N$ are the uniform space stepsize and time steplength, respectively.
Denote $G_{i,j,\rho}^n$ and $f_{i,j,\rho}^n$, respectively, as the numerical approximation to $G(x_i,y_j,\rho,t_n)$ and $f(x_i,y_j,\rho,t_n)$.
To approximate (\ref{1.2}), we utilize the second order central difference formula for the spatial derivative.
According to (\ref{4.7}) and (\ref{4.2}), then (\ref{1.2}) can be recast as
\begin{equation}\label{4.16}
\begin{split}
&\frac{1}{\tau^\alpha}\sum_{k=0}^{n}{d}_{k}^{\nu,\alpha}\left[G(x_i,y_j,\rho,t_{n-k})-e^{-\rho (n-k)\tau}G(x_i,y_j,\rho,0)\right]\\
&=\kappa_\alpha \left(\delta_x^2 G(x_i,y_j,\rho,t_n)+\delta_y^2 G(x_i,y_j,\rho,t_n)\right)
+f(x_i,y_j,\rho,t_{n})+\overline{r}_i^n,
\end{split}
\end{equation}
with the local truncation error $\overline{r}_i^n=\mathcal{O}(\tau^\nu+h^2),~ \nu=1,2,3,4$.
Then, the resulting discretization of (\ref{4.16}) has the following form
\begin{equation}\label{4.17}
\begin{split}
&{l}_{0}^{\nu,\alpha}G_{i,j,\rho}^{n}
+\mu_{h,\tau}^\alpha\left(- G_{i,j-1,\rho}^{n}- G_{i-1,j,\rho}^{n}+4G_{i,j,\rho}^{n}-G_{i+1,j,\rho}^{n}-G_{i,j+1,\rho}^{n}\right)\\
&=-\sum_{k=1}^{n-1}e^{-\rho  k \tau}{l}_k^{\nu,\alpha}
G_{i,j,\rho}^{n-k}
 +\sum_{k=0}^{n-1}e^{-\rho  n \tau}{l}_k^{\nu,\alpha}
G_{i,j,\rho}^{0}
 +\tau^\alpha f_{i,j,\rho}^{n}
\end{split}
\end{equation}
with $\mu_{h,\tau}^\alpha =\kappa_\alpha \frac{\tau^\alpha}{h^2}$.
Denote the  grid functions
 \begin{equation*}
{\bf G}^{n}=[{\bf G}_{1}^n,{\bf G}_{2}^n,\ldots,{\bf G}_{M}^n]^{\rm T}
 ~~{\rm and}~~
{\bf f}^{n}=[{\bf f}_{1}^n,{\bf f}_{2}^n,\ldots,{\bf f}_{M}^n]^{\rm T},
  \end{equation*}
where
 \begin{equation*}
{\bf G}_i^{n}=[{\bf G}_{i,1,\rho}^n,{\bf G}_{i,2,\rho}^n,\ldots,{\bf G}_{i,M,\rho}^n]^{\rm T}
 ~~{\rm and}~~
{\bf f}_i^{n}=[{\bf f}_{i,1,\rho}^n,{\bf f}_{i,2,\rho}^n,\ldots,{\bf f}_{i,M,\rho}^n]^{\rm T}.
  \end{equation*}
For simplicity,  the zero boundary conditions are used.
Thus  (\ref{4.17})  reduces to
\begin{equation}\label{4.18}
\begin{split}
&\left[l_{0}^{\nu,\alpha}I\otimes I+ \mu_{h,\tau}^\alpha \left( I\otimes L+L\otimes I\right)\right]
{\bf G}^{n}\\
&=-\sum_{k=1}^{n-1}e^{-\rho  k \tau}{l}_k^{\nu,\alpha}{\bf G}^{n-k}
               +\sum_{k=0}^{n-1}e^{-\rho  n \tau}{l}_k^{\nu,\alpha}{\bf G}^{0}
               +\tau^\alpha {\bf f}^{n}.
\end{split}
\end{equation}

\section{Applications of MGM}
To align the solution of the resulting algebraic  system (\ref{4.14})
with the   Multigrid Algorithm \ref{MGM}, we assume that the  $A_h=l_{0}^{\nu,\alpha}{H}+\mu_{h,\tau}^\alpha L$, $\nu^h=G^n$ and $$f_h=-\sum_{k=1}^{n-1}e^{-\rho  k \tau}{l}_k^{\nu,\alpha}{H}G^{n-k}
               +\sum_{k=0}^{n-1}e^{-\rho  n \tau}{l}_k^{\nu,\alpha}{H}G^{0}
               +\tau^\alpha{H}F^{n}+\widetilde{F}^n.$$
Then the resulting algebraic system (\ref{4.14}) reduces to the form of (\ref{2.1}), i.e.,
\begin{equation}\label{5.1}
  A_h\nu^h=f_h~~~{\rm with}~~~A_h=l_{0}^{\nu,\alpha}{H}+\mu_{h,\tau}^\alpha L.
\end{equation}

\begin{lemma}\label{lemma5.1}
Let $A^{(1)}:=A_h$ be defined by
(\ref{5.1}) and $A^{(k)}=I_k^{k-1}A^{(k-1)}I_{k-1}^{k}$. Then
\begin{equation*}
\frac{\omega}{\lambda_{\max}(A_k) }(\nu^k,\nu^k) \leq (S_k\nu^k,\nu^k)\leq (A_k^{-1}\nu^k,\nu^k),
~~~\forall \nu^k\in \mathcal{M}_k,
\end{equation*}
where $A_k=A^{(K-k+1)}$, $S_k=\omega D_k^{-1}$, $\omega \in (0,1/2]$ and $D_k$ is the diagonal of $A_k$.
\end{lemma}
\begin{proof}
According to Lemma \ref{lemma2.6} and  $A_h=l_{0}^{\nu,\alpha}{H}+\mu_{h,\tau}^\alpha L$ in (\ref{4.14}), the desired result is obtained.
\end{proof}

\begin{lemma}\label{lemma5.2}
Let $A^{(1)}:=A_h$ be defined by
(\ref{5.1}) and $A^{(k)}=I_k^{k-1}A^{(k-1)}I_{k-1}^{k}$. Then
\begin{equation*}
   \min_{\nu^{k-1} \in \mathcal{B}_{k-1} }||\nu^k-I_{k-1}^k\nu^{k-1}||_{A_k}^2\leq 16 ||A_k\nu^k||_{D_k^{-1}}^2 \quad  \forall \nu^k \in \mathcal{B}_k
\end{equation*}
with $A_k=A^{(K-k+1)}$.
\end{lemma}

\begin{proof}
Since   $A_h=l_{0}^{\nu,\alpha}{H}+\mu_{h,\tau}^\alpha L$ in (\ref{4.14}),  i.e.,
$$a_0=\frac{10}{12}{l}_{0}^{\nu,\alpha}+2\mu_{h,\tau}^\alpha ~~{\rm and}~~
a_1=\frac{1}{12}{l}_{0}^{\nu,\alpha}-\mu_{h,\tau}^\alpha. $$
Combining Lemma  \ref{lemma2.8} and  that $\forall k\geq 1$, there exists
\begin{equation*}
\begin{split}
\frac{\left(6C_k+2^{k-1}\right)\left(a_0+2a_1\right)}{\left(2C_k+2^{k-1}\right)\left(a_0+2a_1\right)-2^{k+1}a_1}
=\frac{\left(6C_k+2^{k-1}\right)l_{0}^{\nu,\alpha}}{\left(2C_k+\frac{2}{3}\cdot2^{k-1}\right)l_{0}^{\nu,\alpha}+2^{k+1}\mu_{h,\tau}^\alpha}
<3,
\end{split}
\end{equation*}
leads to the desired result.
\end{proof}

From Lemmas \ref{lemma5.1},  \ref{lemma5.2} and Theorem \ref{theorem2.9}, our MGM convergence result is obtained.
\begin{theorem}\label{theorem5.3}
For the resulting algebraic system (\ref{4.14}),  it satisfies
$$||I-B_kA_k||_{A_k} \leq \frac{16}{2l\omega+16}<1~~{\rm with}~~~1\leq k\leq K, ~~~~\omega \in (0,1/2],$$
where  the operator $B_k$ is defined by the  V-cycle method in  Multigrid Algorithm  \ref{MGM}  and  $l$ is the number of smoothing steps.
\end{theorem}

According to Theorem \ref{theorem3.8}, for the two-dimensional fractional Feynman-Kac equation, we have the following results.
\begin{theorem}\label{theorem5.4}
For the resulting algebraic system (\ref{4.18}),  it satisfies
$$||{\bf I}-{\bf B}_k{\bf A}_k||_{{\bf A}_k} \leq \frac{1536}{2l\omega+1536}<1~~{\rm with}~~~1\leq k\leq K, ~~~~\omega \in (0,1/4],$$
where  the operator ${\bf B}_k$ is defined by the  V-cycle method in  Multigrid Algorithm  \ref{MGM}  and  $l$ is the number of smoothing steps.
\end{theorem}

\section{Numerical Results}
We employ the V-cycle MGM  described in Algorithm  \ref{MGM} to solve the resulting system.
The stopping criterion is taken as
$$\frac{||r^{(i)||}}{||r^{(0)}||}<10^{-11}~~ {\rm for }~~(\ref{4.14}),~~~~~~~~  \frac{||r^{(i)||}}{||r^{(0)}||}<10^{-7}~~ {\rm for }~~(\ref{4.18}),$$
where $r^{(i)}$ is the residual vector after $i$ iterations;
and the  number of  iterations $(m_1,m_2)=(1,2)$ and $(\omega_{pre},\omega_{post})=(1,1/2).$
In all tables, $M$  denotes the number of spatial grid point;  the numerical errors are measured by the $ l_{\infty}$
(maximum) norm; and  `Rate' denotes the convergent orders.
`CPU' denotes the total CPU time in seconds (s) for solving the resulting discretized   systems;
and `Iter' denotes the average number of iterations required to solve a general linear system $A_h\nu^h=f_h$ at each time level.

All the computations are carried out on a PC with the configuration:
Intel(R) Core(TM) i5-3470 3.20 GHZ and 8 GB RAM and a 64 bit Windows 7 operating system.
Example \ref{example6.1} and  \ref{example6.2} numerical experiments are, respectively,  in Matlab and in Python.

\begin{example}\label{example6.1}
Consider the fractional Feynman-Kac equation  (\ref{4.1}) for 1D, on a finite domain  $0< x < 1 $,  $0<t \leq 1$
with the coefficient
$\kappa_\alpha=1$,  $\rho=1+ \sqrt{-1}$,  the forcing function
\begin{equation*}
\begin{split}
f(x,\rho,t)=& \frac{\Gamma(5+\alpha)}{\Gamma(5)}   e^{-\rho t}t^4(\sin(\pi x)+1)
        +\kappa_\alpha\pi^2 e^{-\rho t}(t^{4+\alpha}+1)\sin(\pi x) ,
\end{split}
\end{equation*}
the initial condition $G(x,\rho,0)=\sin(\pi x)+1$,  and the boundary
conditions $G(0,\rho,t)=G(1,\rho,t)=e^{-\rho t}(t^{4+\alpha}+1)$. Then (\ref{4.1}) has the exact
solution $$G(x,\rho,t)= e^{-\rho t}(t^{4+\alpha}+1)(\sin(\pi x)+1). $$
\end{example}

\begin{table}[h]\fontsize{9.5pt}{12pt}\selectfont
  \begin{center}
  \caption{MGM to solve (\ref{4.14}) at $T=1$ with $\nu=4$, $h=1/M$  and $N=M$,
  where $A_{k-1}=I_k^{k-1}A_kI_{k-1}^{k}$ ({\em Galerkin approach or algebraic MGM}) is computed by (\ref{2.13}).}\vspace{5pt}
 {\small   \begin{tabular*}{\linewidth}{@{\extracolsep{\fill}}*{9}{c}}                                    \hline  
$M$            &  $\alpha=0.3$   & Rate    & Iter    & CPU       &  $\alpha=0.8$  &   Rate &  Iter  &  CPU    \\\hline
    $2^{5}$   &   4.2225e-07  &           &  10     & 0.21 s   &  1.3008e-06         &          &  9    & 0.18 s\\\hline 
    $2^{6}$   &   2.6394e-08  & 3.9998    &  10     & 0.52 s   &  8.1345e-08         &  3.9992  &  9    & 0.47 s \\\hline 
    $2^{7}$   &   1.6494e-09  & 4.0002    &  10     & 1.41 s   &  5.0850e-09         &  3.9997  &  9    & 1.30 s \\\hline 
    $2^{8}$   &   1.0381e-10  & 3.9899    &  10     & 3.97 s   &  3.1723e-10         &  4.0026  &  9    & 3.71 s \\\hline 
    \end{tabular*}}\label{tab:1}
  \end{center}
\end{table}

\begin{table}[h]\fontsize{9.5pt}{12pt}\selectfont
  \begin{center}
  \caption{MGM to solve (\ref{4.14}) at $T=1$ with $\nu=4$, $h=1/M$  and $N=M$,
  where $A_{k-1}=l_0^{\nu,\alpha}{H}+\mu_{2^{K-k+1}h,\tau}^\alpha L$ ({\em doubling the mesh size or  geometric MGM}) is defined by (\ref{4.14})
  .}\vspace{5pt}
 {\small   \begin{tabular*}{\linewidth}{@{\extracolsep{\fill}}*{9}{c}}                                    \hline  
$M$            &  $\alpha=0.3$   & Rate    & Iter    & CPU       &  $\alpha=0.8$  &   Rate &  Iter  &  CPU    \\\hline
    $2^{5}$   &   4.2225e-07  &           &  10     & 0.19 s   &  1.3008e-06         &          &  9    & 0.19 s\\\hline 
    $2^{6}$   &   2.6394e-08  & 3.9998    &  10     & 0.48 s   &  8.1345e-08         &  3.9992  &  9    & 0.45 s \\\hline 
    $2^{7}$   &   1.6498e-09  & 3.9998    &  10     & 1.32 s   &  5.0851e-09         &  3.9997  &  9    & 1.25 s \\\hline 
    $2^{8}$   &   1.0396e-10  & 3.9882    &  10     & 3.62 s   &  3.1730e-10         &  4.0024  &  9    & 3.60 s \\\hline 
    \end{tabular*}}\label{tab:2}
  \end{center}
\end{table}
We use two coarsening strategies: {\em Galerkin approach} and {\em doubling the mesh size}, respectively, to solve the
resulting system (\ref{5.1}).  Tables \ref{tab:1} and \ref{tab:2} show that these two methods have almost the same error values
with the global truncation error $\mathcal{O}(\tau^\nu+h^4),~ \nu=4$,  so that  the locally weighted averaging of Galerkin approach brings convenience for handling the convergence proof, but no more benefits are obtained. In fact, as proved in \cite{doatelliNLAA}, the convergence conditions of the Galerkin and of the geometric approaches are very similar (except for the full rank of the projector which is needed in the Galerkin approach only). However, in general, the Galerkin technique is more robust and the potential reason for which here this fact is not observed is  the presence of the stiffness matrix which improves the conditioning of the problem, acting as a mild regularization.

\begin{example}\label{example6.2}
Consider the fractional Feynman-Kac equation  (\ref{4.1}) for 2D, on a finite domain  $0< x,y < 1 $,  $0<t \leq 1$
with the coefficient $\kappa_\alpha=1$,  $\rho=1$,  the initial condition is  $G(x,\rho,0)=0$  and the zero boundary
conditions on the rectangle. Taking  the exact
solution as $$G(x,y,\rho,t)= e^{-\rho t}t^{4+\alpha}\sin(\pi x)\sin(\pi y)$$
and using above assumptions, it is easy to obtain the forcing functions $f(x,y,\rho,t)$.
\end{example}
\begin{table}[h]\fontsize{9.5pt}{12pt}\selectfont
  \begin{center}
  \caption{MGM to solve (\ref{4.18}) at $T=1$ with $\nu=2$, $h=1/M$  and $N=M$, where
   $A_{k-1}=l_{0}^{\nu,\alpha}I\otimes I+ \mu_{2^{K-k+1}h,\tau}^\alpha \left( I\otimes L+L\otimes I\right)$  is defined by (\ref{4.18}).}\vspace{5pt}
 {\small   \begin{tabular*}{\linewidth}{@{\extracolsep{\fill}}*{9}{c}}                                    \hline  
$M$            &  $\alpha=0.3$   & Rate    & Iter    & CPU       &  $\alpha=0.8$  &   Rate &  Iter  &  CPU    \\\hline
    $2^{4}$   &   1.4647e-03  &           &  17     & 2.01 s   &   2.0068e-03         &   &  16    & 2.00 s \\\hline 
    $2^{5}$   &   3.3496e-04  & 2.1285    &  17     & 9.45 s   & 4.6874e-04         &   2.098  &  16    & 8.84 s \\\hline 
    $2^{6}$   &   8.0048e-05  & 2.0650    &  18     & 44.70 s   &  1.1340e-04         &  2.047  & 15    & 36.48 s \\\hline 
    $2^{7}$   &   1.9564e-05  & 2.0327    &  18     & 193.65 s   &  2.7896e-05         &  2.023  &  15    & 162.44 s \\\hline 
    \end{tabular*}}\label{tab:3}
  \end{center}
\end{table}
From Table \ref{tab:3}, we numerically confirm that the numerical scheme has second-order accuracy in both time and space directions.

\begin{remark}
Since the joint PDF $G(x,A,t)$ is the inverse Laplacian transform $\rho\rightarrow A$ of $G(x,\rho,t)$, for getting $G(x,A,t)$, we need to further perform the inverse numerical Laplacian transform, which has been discussed in \cite{Deng:14}.
\end{remark}

\section{Concluding remarks and future work}
This paper provides few ideas for verifying the uniform convergence  of the V-cycle MGM for symmetric positive definite Toeplitz block tridiagonal matrices, where we use the
simple (traditional) restriction operator and prolongation operator to handle
 general Toeplitz systems directly  for the elliptic PDEs.
Then we further derive the difference scheme for the backward fractional Feynman-Kac equation, which describes the distribution of the functional of non-Brownian particles;
finally, the V-cycle multigrid method is effectively used to solve the generated algebraic system, and the uniform convergence is obtained. In particular, for the coarsing of multigrid methods, even though the geometric MGM and algebraic MGM are different in theoretical analysis and techniques, numerically most of the time almost the same numerical results can be got.
Concerning the future work, the main point to investigate is the extension of this proof  to general banded or dense Toeplitz matrices \cite{Arico:07,Arico:04}.
In fact, for the full Toeplitz matrices with a  weakly diagonally dominant symmetric  Toeplitz  M-matrices,  the condition  (\ref{2.10}) holds when $\omega \in (0,1/3]$ \cite{Chen:16}. Hence the real challenge is the verification of  condition (\ref{2.11}) or of condition (\ref{2.12}) and this will the subject of future researches.

\section*{Acknowledgments}
The first author wishes to thank Qiang Du for his valuable comments while working in Columbia university.
This work was supported by NSFC 11601206 and 11671182, the Fundamental Research Funds for the Central Universities under Grant No. lzujbky-2016-105.

\section*{Appendix}
{\it\bf Proof of    Lemma \ref{lemma2.4}}.
Since $A^{(k)}$ is the  symmetric matrix,
we denote
$A^{(k)}=\{a_{i,j}^{(k)}\}_{i,j=1}^{\infty}$  with $a_{i,j}^{(k)}=a_{|i-j|}^{(k)}~~~~\forall k \geq 1.$
Using the  relation $A^{(k)}=L_h^{H}A^{(k-1)}L_{H}^{h}$, there exists
$$\{b_{j,l}^{(k)}\}_{j,l=1}^{\infty}=A^{(k-1)}L_{H}^{h}~~~~{\rm and}~~~~\{a_{i,l}^{(k)}\}_{i,l=1}^{\infty}=L_h^{H}A^{(k-1)}L_{H}^{h}$$
with $b_{j,l}^{(k)}=a_{2l-j-1}^{(k-1)}+2a_{2l-j}^{(k-1)}+a_{2l-j+1}^{(k-1)}$ and
$a_{i,l}^{(k)}=b_{2i-1,l}^{(k)}+2b_{2i,l}^{(k)}+b_{2i+1,l}^{(k)}$.
Then for the Toeplitz matrix $A^{(k)}$, it holds
\begin{equation}\tag{A.1}
\begin{split}
a_0^{(k)}&=6a_0^{(k-1)}+8a_1^{(k-1)}+2a_2^{(k-1)} \quad \forall k \geq 2; \\
a_j^{(k)}&=a_{2j-2}^{(k-1)} +4a_{2j-1}^{(k-1)}+6a_{2j}^{(k-1)}+4a_{2j+1}^{(k-1)}+a_{2j+2}^{(k-1)} \quad \forall j \geq 1.
\end{split}
\end{equation}
We prove  (\ref{2.13}) by mathematical induction.
For $k=2$, Eq. (\ref{2.13}) holds obviously. Suppose (\ref{2.13}) holds for $k=2,3,\ldots s$.  In particular,   for $k=s$, we have
\begin{equation}\tag{A.2}
\begin{split}
a_0^{(s)}
=&(4C_s+2^{s-1})a_0^{(1)}+\sum_{m=1}^{2\cdot2^{s-1}-1}{_0}C_m^sa_m^{(1)};\\
a_1^{(s)}
=&C_sa_0^{(1)}+\sum_{m=1}^{3\cdot2^{s-1}-1}{_1}C_m^s a_m^{(1)};\\
a_j^{(s)}
=&\sum_{m=(j-2)2^{s-1}}^{(j+2)2^{s-1}-1} {_j}C_m^s a_m^{(1)} \quad \forall j\geq 2.
    \end{split}
\end{equation}
Next we need to  prove that  (\ref{2.13}) holds for $k=s+1$.

According to (A.1), (A.2) and the coefficients  ${_j}C_m^s$, $j\geq 0$ in (\ref{2.13}), we can  check that
\begin{equation*}
\begin{split}
a_0^{(s+1)}
&=6a_0^{(s)}+8a_1^{(s)}+2a_2^{(s)}\\
&=\left(32c_s+6\cdot 2^{s-1}\right)a_0^{(1)}+\sum_{m=1}^{2^{s-1}-1}\left(6\cdot{_0}C_m^{s} +8\cdot{_1}C_m^{s}+2\cdot{_2}C_m^{s} \right)a_m^{(1)}\\
&\quad+\sum_{m=2^{s-1}}^{2\cdot2^{s-1}-1}\left(6\cdot{_0}C_m^{s} +8\cdot{_1}C_m^{s}+2\cdot{_2}C_m^{s} \right)a_m^{(1)}\\
&\quad+\sum_{m=2\cdot2^{s-1}}^{3\cdot2^{s-1}-1}\left(8\cdot{_1}C_m^{s}+2\cdot{_2}C_m^{s} \right)a_m^{(1)}
      +\sum_{m=3\cdot2^{s-1}}^{4\cdot2^{s-1}-1}2\cdot{_2}C_m^{s} a_m^{(1)}\\
&=(4C_{s+1}+2^{s})a_0^{(1)}+\sum_{m=1}^{4\cdot2^{s-1}-1}{_0}C_m^{s+1}a_m^{(1)};
    \end{split}
\end{equation*}
\begin{equation*}
\begin{split}
a_1^{(s+1)}
&=a_0^{(s)}+4a_1^{(s)}+6a_2^{(s)}+4a_3^{(s)}+a_4^{(s)}\\
&=\left(8c_s+ 2^{s-1}\right)a_0^{(1)}+\sum_{m=1}^{2^{s-1}-1}\left({_0}C_m^{s} +4\cdot{_1}C_m^{s}+6\cdot{_2}C_m^{s} \right)a_m^{(1)}\\
&\quad+\sum_{m=2^{s-1}}^{2\cdot2^{s-1}-1}\left({_0}C_m^{s} +4\cdot{_1}C_m^{s}+6\cdot{_2}C_m^{s} +4\cdot{_3}C_m^{s} \right)a_m^{(1)}\\
&\quad+\sum_{m=2\cdot2^{s-1}}^{3\cdot2^{s-1}-1}\left(4\cdot{_1}C_m^{s}+6\cdot{_2}C_m^{s} +4\cdot{_3}C_m^{s}+{_4}C_m^{s} \right)a_m^{(1)}\\
&\quad+\sum_{m=3\cdot2^{s-1}}^{4\cdot2^{s-1}-1}\left(6\cdot{_2}C_m^{s} +4\cdot{_3}C_m^{s}+{_4}C_m^{s} \right)a_m^{(1)}\\
&\quad+\sum_{m=4\cdot2^{s-1}}^{5\cdot2^{s-1}-1}\left(4\cdot{_3}C_m^{s}+{_4}C_m^{s} \right)a_m^{(1)}
 +\sum_{m=5\cdot2^{s-1}}^{6\cdot2^{s-1}-1}{_4}C_m^{s} a_m^{(1)}\\
&=C_{s+1}a_0^{(1)}+\sum_{m=1}^{6\cdot2^{s-1}-1}{_1}C_m^{s+1} a_m^{(1)};\\
    \end{split}
\end{equation*}
and
\begin{equation*}
\begin{split}
a_j^{(s+1)}
&=a_{2j-2}^{(s)} +4a_{2j-1}^{(s)}+6a_{2j}^{(s)}+4a_{2j+1}^{(s)}+a_{2j+2}^{(s)}\\
&=\sum_{m=(2j-4)2^{s-1}}^{(2j-3)2^{s-1}-1} {_{2j-2}}C_m^{s} a_m^{(1)}
+\sum_{m=(2j-3)2^{s-1}}^{(2j-2)2^{s-1}-1} \left(  {_{2j-2}}C_m^{s} +4\cdot{_{2j-1}}C_m^{s}   \right)a_m^{(1)}\\
&\quad+\sum_{m=(2j-2)2^{s-1}}^{(2j-1)2^{s-1}-1} \left(  {_{2j-2}}C_m^{s} +4\cdot{_{2j-1}}C_m^{s}+6\cdot{_{2j}}C_m^{s}     \right)a_m^{(1)}\\
&\quad+\sum_{m=(2j-1)2^{s-1}}^{2j\cdot2^{s-1}-1} \left(  {_{2j-2}}C_m^{s} +4\cdot{_{2j-1}}C_m^{s}+6\cdot{_{2j}}C_m^{s}+4\cdot{_{2j+1}}C_m^{s}      \right)a_m^{(1)}\\
&\quad+\sum_{m=2j\cdot2^{s-1}}^{(2j+1)\cdot2^{s-1}-1} \left(  4\cdot{_{2j-1}}C_m^{s}+6\cdot{_{2j}}C_m^{s}+4\cdot{_{2j+1}}C_m^{s}+{_{2j+2}}C_m^{s}       \right)a_m^{(1)}\\
&\quad+\sum_{m=(2j+1)\cdot2^{s-1}}^{(2j+2)\cdot2^{s-1}-1} \left(  6\cdot{_{2j}}C_m^{s}+4\cdot{_{2j+1}}C_m^{s}+{_{2j+2}}C_m^{s}       \right)a_m^{(1)}\\
&\quad+\sum_{m=(2j+2)\cdot2^{s-1}}^{(2j+3)\cdot2^{s-1}-1} \left(4\cdot{_{2j+1}}C_m^{s}+{_{2j+2}}C_m^{s}       \right)a_m^{(1)}
 +\sum_{m=(2j+3)\cdot2^{s-1}}^{(2j+4)\cdot2^{s-1}-1} {_{2j+2}}C_m^{s}       a_m^{(1)}\\
&=\sum_{m=(2j-4)2^{s-1}}^{(2j+4)2^{s-1}-1} {_j}C_m^{s+1} a_m^{(1)}.
    \end{split}
\end{equation*}
The proof is completed.


\begin{thebibliography}{99}

\bibitem{Arico:07} Aric\`{o},  A.,    Donatelli, M.:    A V-cycle multigrid for multilevel matrix algebras: proof of optimality.
Numer. Math.  \textbf{105},   511--547   (2007).

\bibitem{Arico:04}  Aric\`{o},  A.,    Donatelli, M.,   Serra-Capizzano, S.:
V-cycle optimal convergence for certain (multilevel) structured linear systems.
SIAM J. Matrix Anal. Appl.  \textbf{26},  186--214   (2004).


\bibitem{Bar-Haim:98}  Bar-Haim, A.,  Klafter,  J.:  On mean residence and first passage times in finite
one-dimensional systems. J. Chem. Phys., \textbf{109}, 5187--5193  (1998).


\bibitem{Bank:85}    Bank, R.E.,   Douglas,  C.C.:
 Sharp estimates for multigrid rates of convergence with general smoothing and acceleration.  SIAM J. Numer. Anal. {\bf22},  617--633   (1985).


\bibitem{Bolten:15}   Bolten, M.,     Donatelli, M.,    Huckle, T.,    Kravvaritis,  C.:
Generalized grid transfer operators for multigrid methods applied on Toeplitz matrices.  BIT., {\bf55},  341--366    (2015).

\bibitem{Bramble:87}Bramble, J.H.,   Pasciak,  J.E.:
 New convergence estimates for multigrid algorithms.  Math. Comp.  \textbf{49},  311--329 (1987).

\bibitem{Bramble:91}Bramble, J.H.,   Pasciak,  J.E.,  Wang, J.P.,   Xu, J.H.:
Convergence estimates for multigrid algorithms without regularity assumptions. Math. Comp. \textbf{57},  23--45  (1991).



\bibitem{Brenner:08}  Brenner, S.C.,  Scott, L.R.: The Mathematical Theorey of Finite Element Methods.  Springer, (2008).


\bibitem{Carmi:10}  Carmi, S., Turgeman,  L.,   Barkai, E.: On distributions of functionals of anomalous diffusion paths.
J. Stat. Phys. \textbf{141},  1071--1092  (2010).

\bibitem{Chan:98}Chan, R.H.,   Chang, Q.S., Sun,  H.W.:  Multigrid method for ill-conditioned symmetric  Toeplitz systems.
 SIAM J. Sci. Comput.  \textbf{19},  516--529  (1998).


\bibitem{Chan:07}  Chan, R.H.,  Jin, X.Q.: An Introduction to Iterative Toeplitz Solvers. SIAM, (2007).




\bibitem{Chen:14}Chen, M.H., Wang, Y.T.,  Cheng,  X.,  Deng, W.H.:
Second-order LOD multigrid method for multidimensional Riesz fractional diffusion equation.
 BIT Numer. Math.  \textbf{54},  623--647  (2014).


\bibitem{Chen:0013}  Chen, M.H.,  Deng, W.H.:
Fourth order accurate scheme for the space fractional diffusion equations.
SIAM J. Numer. Anal.  \textbf{52},  1418--1438   (2014).


\bibitem{Chen:1313} Chen, M.H.,  Deng, W.H.:
Fourth order difference approximations for space Riemann-Liouville derivatives based on weighted and shifted Lubich difference operators.  Commun. Comput. Phys. \textbf{16},  516--540  (2014).



\bibitem{Chen:13}Chen, M.H.,  Deng, W.H.:
  Discretized fractional substantial calculus.  ESAIM: Math. Mod. Numer. Anal.
\textbf{49},  373-394  (2015).

\bibitem{ChenDeng:15} Chen, M.H.,  Deng, W.H.:
 High order algorithms for the fractional substantial diffusion equation with truncated L\'{e}vy flights.
SIAM J.  Sci. Comput.   \textbf{37},  A890--A917  (2015).


\bibitem{Chen:16}Chen, M.H.,  Deng, W.H.:
  Convergence proof for the multigird method of the nonlocal model.  SIAM J. Matrix Anal. Appl. (minor revised),  arXiv:1605.05481.





\bibitem{Deng:14}   Deng, W.H.,   Chen, M.H.,    Barkai, E.:   Numerical algorithms for the forward and backward fractional Feynman-Kac equations. J. Sci. Comput.  \textbf{62}, 718--746  (2015).



\bibitem{doatelliNLAA} Donatelli, M.: An algebraic generalization of local Fourier analysis for grid transfer operators in multigrid based on Toeplitz matrices. Numer. Linear Algebra Appl. \textbf{17}, 179--197 (2010).




\bibitem{Fiorentino:91}   Fiorentino, G.,   Serra,  S.:
Multigrid methods for Toeplitz matrices. Calcolo.  \textbf{28}, 283--305  (1991).


\bibitem{Fiorentino:96}  Fiorentino, G., Serra, S.
 Multigrid methods for symmetric positive definite block Toeplitz matrices with nonnegative generating functions.
SIAM J.  Sci. Comput.  \textbf{17}, 1068--1081  (1996).

\bibitem{Golub:96}   Golub, G.H.,   Van Loan, C.F.:   Matrix Computations.  The Johns Hopkins University Press, (1996).

\bibitem{szego}Grenander,  U.,   Szeg\"o, G.: Toeplitz Forms and Their Applications. Chelsea, New York, (1984).

\bibitem{Hackbusch:85}Hackbusch, W.: Multigird Methods and Applications.  Springer-Verlag, Berlin, (1985).



\bibitem{Horn:13}  Horn,  R.A.,  Johnson,  C.R.: Matrix Analysis.
Cambridge University Press, New York, (2013).



\bibitem{Horton:95}  Horton, G.,    Vandewalle,  S.:   A  space-time multigrid method for parabolic partial differential equations.
SIAM J.  Sci. Comput.  \textbf{16},  848--864  (1995).

\bibitem{Ji:15}  Ji, C.C.,    Sun, Z.Z.:   A  higher-order compact finite difference scheme for the fractional sub-diffusion equation. J. Sci. Comput.  \textbf{64}, 959--985  (2015).



\bibitem{Laub:05} Laub, A.J.:  Matrix Analysis for Scientists and Engineers.  SIAM, (2005).



\bibitem{Meurant:92}  Meurant,  G.:   A review on the inverse of symmetric tridiagonal and block tridiagonal matrices.
SIAM J. Matrix Anal. Appl.  {\bf13},  707--728   (1992).


\bibitem{Pang:12}  Pang, H.,   Sun,  H.:
 Multigrid method for fractional diffusion equations. J. Comput. Phys. \textbf{231}, 693--703  (2012).

\bibitem{Quarteroni:07}  Quarteroni,  A.,    Sacco, R.,     Saleri, F.:   Numerical Mathematics. Springer, (2007).


\bibitem{Ruge:87}  Ruge,  J.,    St\"{u}ben, K.:  Algebraic multigrid, in  Multigrid Methods
Ed:  McCormick S.,  73-130, SIAM, (1987).

\bibitem{Saad:03}  Saad, Y.:  Iterative Methods for Sparse Linear Systems.  SIAM, (2003).

\bibitem{Serra-Capizzano:02}   Serra-Capizzano, S.:
 Convergence analysis of two-grid methods for elliptic Toeplitz and PDEs matrix-sequences.
Numer. Math.  \textbf{92},   433--465 (2002).

\bibitem{Stoer:02}  Stoer, J.,    Bulirsch, R.:  Introduction to Numerical Analysis.  Springer, (2002).




\bibitem{Trottenberg:01}   Trottenberg, U.,   Oosterlee, C.W.,   Sch\"{u}ller, A.:
Multigird.  Academic Press, New York, (2001).


 \bibitem{Turgeman:09}  Turgeman, L.,  Carmi, S.,   Barkai,  E.: Fractional Feynman-Kac equation for non-Brownian functionals.
                      Phys. Rev. Lett.  \textbf{103}, 190201   (2009).



\bibitem{Xu:02}     Xu, J.,   Zikatanov, L.:
The method of alternating projections and the method of subspace corrections in Hilbert space.
J. Am. Math. Soc.  \textbf{15}, 573--597  (2002).





\end{thebibliography}
\end{document}